\documentclass[12pt]{article}
\usepackage{amsmath,amssymb,scrtime,bbm}
\usepackage{amsthm}
\usepackage{epsfig,psfrag,latexsym}
\usepackage{tikz}
\usepackage{hyperref}

\numberwithin{equation}{section}

\newcommand{\E}{\mathbb{E}}
\renewcommand{\P}{\mathbb{P}}
\newcommand{\N}{\mathbb{N}}

\newcommand{\R}{\mathbb{R}}
\newcommand{\eps}{\epsilon}

\newcommand{\p}{{\sf{p}}}
\newcommand{\q}{{\sf{q}}}
\renewcommand{\S}{\mathcal{S}}
\newcommand{\gga}{g_{\gamma}}

\numberwithin{equation}{section}

\newtheorem{theorem}{Theorem}[section]
\newtheorem{lemma}[theorem]{Lemma}

\newtheorem{corollary}[theorem]{Corollary}
\newtheorem{proposition}[theorem]{Proposition}

\theoremstyle{definition}
\newtheorem{definition}[theorem]{Definition}

\newtheorem{remark}[theorem]{Remark}

\newcommand{\textandreference}[2]{\texorpdfstring{\hyperref[#2]{#1\ref*{#2}}}{#1\ref*{#2}}}
\newcommand{\refsect}[1]{\textandreference{Section~}{#1}}
\newcommand{\refsubsect}[1]{\textandreference{Section~}{#1}}
\newcommand{\refthm}[1]{\textandreference{Theorem~}{#1}}
\newcommand{\refprop}[1]{\textandreference{Proposition~}{#1}}
\newcommand{\reflemma}[1]{\textandreference{Lemma~}{#1}}
\newcommand{\refcoro}[1]{\textandreference{Corollary~}{#1}}

\title{Short paths for first passage percolation\\ on the complete graph}
\author{Maren Eckhoff\thanks{Department of Mathematical Sciences, University of Bath, Bath, BA2 7AY, United Kingdom. Email: {\tt m.eckhoff@bath.ac.uk}}
\and
Jesse Goodman\thanks{Mathematisch Instituut, Universiteit Leiden, P.O.~Box 9512, 2300~RA~Leiden, The~Netherlands. Email: {\tt goodmanja@math.leidenuniv.nl}}
\and Remco van der Hofstad\thanks{Department of Mathematics and Computer Science,
Eindhoven University of Technology, P.O.\ Box  513,
5600 MB Eindhoven, The Netherlands. Email:
{\tt rhofstad@win.tue.nl, f.r.nardi@tue.nl}}\and Francesca R. Nardi$^\ddagger$}
\begin{document}
\maketitle

\begin{abstract}
We study the complete graph equipped with a topology induced by independent and identically distributed edge weights. The focus of our analysis is on the weight $W_n$ and the number of edges $H_n$ of the minimal weight path between two distinct vertices in the weak disorder regime. We establish novel and simple first and second moment methods using path counting to derive first order asymptotics for the considered quantities.
Our results are stated in terms of a sequence of parameters $(s_n)_{n\in \N}$ that quantifies the extreme-value behaviour of the edge weights,
and that describes different universality classes for first passage percolation on the complete graph.
These classes contain both $n$-independent and $n$-dependent edge weight distributions.
 The method is most effective for the universality class containing the edge weights $E^{s_n}$, where $E$ is an exponential$(1)$ random variable and $s_n\log n \to \infty$, $s_n^2 \log n \to 0$. We discuss two types of examples from this class in detail. In addition, the class where $s_n \log n$ stays finite is studied. This article is a contribution to the program initiated in \cite{BhaHof12}.

\end{abstract}



\section{Introduction and Results}
The immense increase of data collection in recent years triggered an interdisciplinary effort to formulate mathematical models that describe real-world networks. The rigorous analysis of properties of these models and the study of asymptotics as the size of the network becomes large is one central theme of modern probability. Important examples of large networks include transportation, data transmission and gene regulatory networks. These networks are not only described by their graph structure, which contains the information about connections between vertices in the network and the degree of separation, but rather by edge weights representing the cost or time required to traverse the edges. In the case of a data network like the Internet, computers are vertices, cables are edges and the weights model the cost needed to transfer a data packet. To describe a transportation network, one could think of stations as vertices, rails as edges, and the time or economic cost to carry a commodity from one station to the next is captured by the edge weights.

The problem of finding the optimal path in a network can be modeled as follows: Let $\mathcal{G}=({\rm V, E})$ be a finite connected (deterministic or random) graph. To each edge $e \in \text{E}$ assign a random edge weight $X_{e}$, where $(X_e)_{e \in \text{E}}$ are positive, independent and identically distributed random variables. We choose uniformly at random two vertices in ${\rm V}$ and label them $1$ and $2$. The set of all self-avoiding paths between vertex $1$ and $2$ is denoted by $\S_{1,2}$. The cost of a path $\p$ between them is a function $w(\p)$ of the edge weights on the path, and the optimal path $\p_{\text{opt}}$ is the path that minimizes $w$ over $\S_{1,2}$.

Two statistics of the optimal path are of particular importance: The actual cost of traversing the path $w(\p_{\text{opt}})$ and, when the minimizer $\p_{\text{opt}}$ is unique, the number $H(\p_{\text{opt}})$ of edges in the optimal path. The latter quantity is called the {\it{hopcount}}.

In the study of random disordered systems, two cost regimes are of great interest. The conventional setup is the {\bf{weak disorder regime}}, where the weight of a path $\p$ is given by
\begin{equation}\label{definitionWnused}
w(\p)=\sum_{e \in \p} X_e.
\end{equation}
Here every edge adds to the weight. In contrast, the {\bf{strong disorder regime}} is determined only by the maximal edge weight contained in the path, i.e.,
\[
w_{\max}(\p):= \max_{e \in \p} X_e.
\]
The weight function corresponding to the {\bf{graph distance}} is given by
\[
w_{\text{graph}}(\p):= \sum_{e \in \p}1.
\]
The focus of our study is on the case of a complete graph $\mathcal{G}=K_n=([n],\text{E}_n)$ where the vertex set equals $[n]=\{1,\dotsc,n\}$. Let $(F_n)_{n \in \N}$ be a sequence of continuous distribution functions concentrated on the positive half-line. For $n \in \N$, the edge weights $X_{e,n}$, $e\in \text{E}_n$, follow distribution $F_n$. It will always be clear from the context which $n$ we consider and we write $X_e$ instead of $X_{e,n}$ to simplify notation. As $F_n$ is continuous, the smallest weight path is unique almost surely and the hopcount $H_n=H(\p_{\text{opt}})$ is well-defined. The weight of the smallest weight path is denoted by $W_n=w(\p_{\text{opt}})$.

The complete graph with random edge weights is a very suitable model to test new techniques which can then be applied in several different contexts. Some of the first results about first passage percolation on this graph can be found in \cite{JAN99} and \cite{HHM01}. Recent progress on related flow problems on the complete graph was achieved in \cite{AMDS09} and \cite{AB10}.

In \cite{BhaHof12} Bhamidi and van der Hofstad initiated a program to study the connection between graph distance, weak disorder and strong disorder. The idea is to introduce a parameter $s \ge 0$ and consider the new weight function
\[
w_s(\p):= \sum_{e \in \p} X_e^s.
\]
Then the classical weak disorder regime corresponds to $s=1$, i.e.~$w_1=w$, and the graph distance can be obtained by choosing $s=0$. The strong disorder regime corresponds to the $s\to\infty$ limit in the sense that the optimal paths agree. To compare the topologies, one has to consider $w_s(\p)^{1/s}$ and take $s \to \infty$. Thus, this model allows us to interpolate between various regimes of interest and to study the phase transitions.

Notice that the model with weight function $w_s$ and edge weights $(X_e)_{e \in \text{E}}$ is equivalent to the model with weight function $w$ in (\ref{definitionWnused}) and edge weights $(X_e^s)_{e \in \text{E}}$. In the sequel, we will work only in the latter setting.

Bhamidi and van der Hofstad \cite{BhaHof12} investigated the case $X_e \overset{d}{=} E^s$, where $E$ denotes an exponentially distributed random variable with mean one and $s>0$ is a constant. They show that the hopcount $H_n$ obeys a central limit
theorem with asymptotic mean $s \log n$ and asymptotic variance $s^2 \log n$. The weight satisfies
\begin{equation}\label{weight_s_const}
\frac{W_n}{n^{-s} s \log(n)} \overset{\P}{\longrightarrow} \frac{1}{s \Gamma(1+1/s)^s} \qquad \text{for }n \to \infty.
\end{equation}
What is more, the weak limit of a linear transformation of $W_n$ is identified. For $s=1$ we obtain the stochastic mean-field model of distance. The first order behaviour of the weight of the optimal path in this model was also discussed in \cite{JAN99}.

In the present article, we extend the study to $n$-dependent $s$, following an idea in \cite{BhaHof12}. We approach the topology induced by the graph distance $w_{\text{graph}}$ by choosing a sequence $(s_n)_{n}$ with $s_n \to 0$. The scaling behaviour of weight $W_n$ and hopcount $H_n$ depends on the speed of $(s_n)$. The main focus of this paper is on the case that $s_n\log n\to \infty$ and $s_n^2\log n\to 0$. This speed of $(s_n)$ describes a universality class of edge weight distributions, as will be explained in more detail in \refsect{discussion}. Our methods are designed to deal with this class and we extend the study to a large family of distributions. The exponential edge weight can be replaced by $X_e\overset{d}{=}Z^{s_n}$ where $Z$ is a continuous, positive random variable. In addition, we discuss the edge weight distribution $X_e\overset{d}{=} e^{-(E/\rho)^{1/\alpha}}$ and show that it falls in the \emph{same} universality class when $\rho >0$ and $\alpha>2$.
As another contribution to the program of \cite{BhaHof12}, the case where $s_n\log n$ stays finite is considered for $X_e\overset{d}{=}Z^{s_n}$.

All limits in this paper are taken as $n$ tends to infinity unless stated otherwise. We say that a sequence of events $(A_n)_{n \in \N}$ occurs \emph{with high probability (whp)} if $\P(A_n) \to 1$. For sequences $(a_n)_n$, $(b_n)_n$, we write $a_n \sim b_n$ when $a_n/b_n \to 1$. The notation $a_n \approx b_n$ is used for heuristic statements only.

\subsection{Results}\label{main_results_sec}

Recall that $F_n$ denotes the continuous distribution function of the edge weights on $K_n$. In accordance with extreme value theory, we denote by $u_n$ the unique positive value with
\begin{equation}\label{defun}
n F_n(u_n)=1.
\end{equation}

Let $Z$ be a positive random variable with distribution function $G$. We assume that there exists some $\hat{x}>0$ such that $G \in C^2([0,\hat{x}])$ is strictly increasing with $\lambda:=G'(0+)>0$. Main examples for distributions satisfying these assumptions are the exponential($\lambda$) distribution and the uniform distribution on $(0,1/\lambda)$. Let $(s_n)_{n \in \N}\in (0,\infty)^{\N}$ be a null sequence.

\begin{theorem}{\bf{(Optimal path: weight and hopcount)}} \label{weight_thm}
Let $X_{e}\overset{d}{=} Z^{s_n}$ with $s_n \log n \to \infty$ and $s_n^2 \log n\to 0$. Then, with high probability,
\begin{equation}\label{weight_bounds_eq}
(1-\eps_n) e\lfloor s_n \log n \rfloor u_n \le W_n \le e\lceil s_n \log n \rceil u_n,
\end{equation}
where $\eps_n=\sqrt{s_n \log \log n}=o(1)$. In particular,
\begin{equation}\label{Wn_in_prob_eq}
\frac{W_n}{u_n s_n \log n} \overset{\P}{\longrightarrow} e \quad\qquad \text{for }n\to \infty.
\end{equation}
Moreover,
\begin{equation}\label{hop_thm}
\frac{H_n}{s_n \, \log n} \overset{\P}{\longrightarrow} 1\qquad \text{for }n\to \infty.
\end{equation}
\end{theorem}

The lower bound in (\ref{weight_bounds_eq}) does not need the assumption $s_n^2 \log n \to 0$. A discussion about this and further cases can be found in \refsect{discussion}. For extensions of \refthm{weight_thm} and a concentration result for the hopcount under additional assumptions see \refsect{main_proofs_sec}.

The next theorem identifies a further distribution which lies in the same universality class as $Z^{s_n}$. The connection between the distributions will be explained heuristically in \refsect{discussion}. This discussion will also clarify our choice of parameters.

\begin{theorem}\label{weight_a_thm}{\bf{($W_n$ and $H_n$ for the $e^{-(E/\rho)^{1/\alpha}}$ case)}}
Let $X_e \overset{d}{=} e^{-(E/\rho)^{1/\alpha}}$ for $\rho >0$, $\alpha >2$. Then
\[
u_n=e^{-(\log n/\rho)^{1/\alpha}}
\]
and, with $s_n=\frac{1}{\alpha \rho^{1/\alpha}} (\log n)^{-1+1/\alpha}$,
\begin{equation}\label{weight_alpha_1}
\frac{W_n}{ u_n s_n \log n} \overset{\mathbb{P}}{\longrightarrow} e \quad \qquad \text{for }n\to \infty.
\end{equation}
Moreover,
\begin{equation}\label{hop_alpha_thm}
\frac{H_n}{s_n \log n} \overset{\mathbb{P}}{\longrightarrow} 1 \qquad \text{for } n \to \infty.
\end{equation}
\end{theorem}

In Theorems \ref{upper_a_thm} and \ref{lower_a_thm}, we shall prove upper and lower bounds on the weight, similar to (\ref{weight_bounds_eq}) in \refthm{weight_thm}, that give more precise information than (\ref{weight_alpha_1}).

One may ask about the behaviour when the condition $s_n \log n \to \infty$ is violated. We study this regime for $X_{e}\overset{d}{=}Z^{s_n}$, where $Z$ is a positive random variable and $G(x)=\P(Z\le x)$ satisfies $G(x)=\lambda x (1+o(1))$ for some $\lambda >0$ and $x \downarrow 0$. In this case the weight of the shortest path no longer tends to zero and the hopcount remains bounded. We are almost back in the graph topology.

\begin{theorem}\label{very_small_sn_thm}{\bf{($W_n$ and $H_n$ for very small $s_n$)}} Let $X_e\overset{d}{=}Z^{s_n}$ with
$s_n \log n \to \gamma \in [0, \infty)$. Then, for every $\eps >0$,
\[
(1-\eps)e^{-\gamma}\le W_n \le 1+\eps \quad \text{and}\quad H_n \le e^{2\gamma}\quad \text{with high probability}.
\]
\end{theorem}

The upper and lower bounds given in \refthm{very_small_sn_thm} are not optimal as can be seen in the next theorem, which is proven under additional assumptions on the distribution function $G$ and the sequence $(s_n)_n$. For $\gamma \in [0,\infty)$ and $x>0$, let
\begin{equation}\label{eq:defgg}
\gga(x)=x e^{-\gamma (1-1/x)}= e^{-\gamma} x e^{\gamma/x}.
\end{equation}
Denote $\gamma_k:=k (k+1) \log \frac{k+1}{k}$ for $k \in \N$ and $\Gamma = \{\gamma_k: k \in \N\}$. If $\gamma \not\in \Gamma$, then the minimizer of $\gga$ on $\N$ is unique and we denote it by $k(\gamma) \in \{\lfloor \gamma\rfloor, \lceil \gamma \rceil\}$. If $\gamma =\gamma_k$ for some $k \in \N$, then $\gga$ has two minimizers in $\N$ given by $k$ and $k+1$. For $\gamma \in [0,2\log 2)$, $k(\gamma)=1$ and $\gga(k(\gamma))=1$. The smallest value in $\Gamma$ is $\gamma = 2\log 2$ yielding $\min_{k \in \N}\gga(k)=1$. For $\gamma >2\log 2$, the minimum of $\gga$ on $\N$ is strictly smaller than one.

For the next theorem we assume that there exists $\hat{x} \in (0,\infty)$ such that $G \in C^2([0,\hat{x}])$ and $\lambda=G'(0+)>0$. Moreover, we write $\Lambda$ for a standard Gumbel random variable, i.e., $\P(\Lambda \le t)=e^{-e^{-t}}$ for all $t \in \R$, and let
\begin{equation}\label{defak}
a_k =\frac{(2\pi)^{k/2}}{\sqrt{2\pi k}} \qquad \forall k \in \N.
\end{equation}

\begin{theorem}
\label{very_small_2nd_order}Let $X_e\overset{d}{=}Z^{s_n}$ with $s_n \log n \to \gamma \in [0,\infty)$.\\
(a) If $\gamma \in [0,2\log 2)$, then
\begin{equation}
\frac{W_n-1}{s_n} \overset{d}{\longrightarrow} \log Z \qquad \text{and}\qquad \P(H_n =1) \rightarrow 1. \label{WnHnVsmalla}
\end{equation}
(b) If $\gamma=2\log 2$, then
\[
W_n \overset{\P}{\longrightarrow} 1 \qquad \text{and} \qquad \P(H_n \in \{1,2\}) \to 1.
\]
(c) If $s_n \log n-\gamma=o((\log n)^{-1/2})$, $\gamma \not \in \Gamma$ and $\gamma >2\log 2$,  then for $k=k(\gamma)$
\begin{align}
&\frac{k}{\gga(k)s_n}\big[\lambda^{s_n} W_n-\gga(k)\big] +\frac{k-1}{s_n} \big(s_n \log n-\gamma-(s_n \log s_n)/2\big) \overset{d}{\longrightarrow} -\Lambda - \log a_k, \label{2ndOrderWnVsmallb}\\
\text{and } \quad& \qquad\P(H_n =k) \rightarrow 1. \label{HnAsympVsmallb}
\end{align}
(d) If $s_n \log n-\gamma=o((\log n)^{-1/2})$ and $\gamma=\gamma_k \in \Gamma$ for some $k \ge 2$, then
\begin{equation} \label{WnHnVsmallSpec}
W_n \overset{\P}{\longrightarrow} \gga(k) \qquad \text{and} \qquad \P(H_n \in \{k,k+1\}) \to 1.
\end{equation}
\end{theorem}

We prove \refthm{very_small_2nd_order} in \refsect{very_small_sn}, using methods similar to those of \cite{BhaHofHoo12pre}.

\subsection{Discussion}\label{discussion}
In this section, we explain the overall picture of our analysis and the key ideas behind methods and proofs. Moreover, we discuss the universality classes of first passage percolation on the complete graph. For a detailed literature review and a discussion of the connection to physical phenomena we refer the 
reader to \cite{BhaHof12} and \cite{BhaHofHoo12pre}.

First, we clarify the type of distributions we have in mind and demonstrate the connection between the considered examples by a simple heuristic.

For finding the shortest path between two vertices, the minimal available edge weight is of interest. By definition of the model, the edge weights $(X_e)_{e\in \text{E}_n}$ are independent and identically distributed according to a strictly increasing, continuous distribution function $F_n$ on $(0,\infty)$. Recall the definition of $u_n$ in (\ref{defun}). It is a well-known result from extreme value theory that $\min_{e \sim v} X_e$ is of scale $u_n$ for any vertex $v\in [n]$. If there is a continuously differentiable function $\Phi_n$ such that
\[
F_n(x)=\exp(-\Phi_n(\log x)) \qquad \text{for all }x > 0 \text{ sufficiently small},
\]
then a first order Taylor expansion implies
\begin{equation}\label{uniform_sn_eq}
F_n(x) \approx \exp\big(-\Phi_n(\log u_n)-\Phi_n'(\log u_n) \log(x/u_n)\big)=\frac{1}{n}\left(\frac{x}{u_n}\right)^{-\Phi_n'(\log u_n)}.
\end{equation}
Our simplest example is $X_e\overset{d}{=}U^{s_n}$ where $U$ is uniformly distributed on $(0,1)$. In this case $u_n =n^{-s_n}$ and $\Phi_n(x)=-x/s_n$. For $x \in (0,1)$ we have
\begin{equation}\label{gen_sn_eq}
F_n(x)=x^{1/s_n}=\frac{1}{n} \left(\frac{x}{u_n}\right)^{1/s_n}.
\end{equation}
Comparing (\ref{uniform_sn_eq}) and (\ref{gen_sn_eq}), we expect that for a general distribution function as above the analogue of $s_n$ is given by
\begin{equation}\label{defsn}
s_n \approx -\frac{1}{\Phi_n'(\log u_n)}.
\end{equation}
The generalization of the uniform distribution we consider is $X_e\overset{d}{=}Z^{s_n}$, where $Z$ is a positive random variable whose distribution function $G(x)=\P(Z\le x)$ is strictly increasing and of class $C^2([0,\hat{x}])$, for some $\hat{x}>0$, with $\lambda=G'(0+)>0$. The distribution function of the edge weights $X_e$ is given by $F_n(x)=G(x^{1/s_n})$. Hence 
\begin{equation}\label{unscaling}
u_n=G^{-1}(1/n)^{s_n}=\Big[\frac{1}{\lambda n} +o\big(1/n\big)\Big]^{s_n} =(\lambda n)^{-s_n} \big[1+o(1)\big]^{s_n}.
\end{equation}
We have $\Phi_n(x)=-\log G(e^{x/s_n})$, and therefore
\[
-\Phi_n'(\log u_n) =\frac{1}{1/n}G'(u_n^{1/s_n}) \frac{1}{s_n} u_n^{1/s_n} \sim \frac{1}{s_n}.
\]
Therefore, the heuristically derived $s_n$ from (\ref{defsn}) agrees asymptotically with the explicit $s_n$ of the distribution. An example where $s_n$ is implicit is given by the edge distribution $X_e\overset{d}{=}e^{-(E/\rho)^{1/\alpha}}$, where $E$ is an exponential random variable with mean one and $\rho >0$, $\alpha >0$. 
Note that the distribution of $X_e$ is \emph{independent} of $n$.
Here
\[
F(x)=\exp\big(-\rho (\log \tfrac{1}{x})^{\alpha}\big) \quad \text{and} \quad u_n=\exp\Big(-\big( \rho^{-1} \log n \big)^{1/\alpha}\Big).
\]
We have $\Phi(x)=\rho (-x)^{\alpha}$, and therefore the implicit $s_n$ is given by 
\begin{equation}\label{implicit_sn_eq}
s_n=\frac{(\log n)^{-1+1/\alpha}}{\alpha\rho^{1/\alpha}}.
\end{equation}
This article mainly focusses on the case that $(s_n)$ is a null sequence with $s_n^2 \log n \to 0$ but $s_n \log n \to \infty$. Since
\[
s_n \log n=\alpha^{-1} \rho^{-1/\alpha} (\log n)^{1/\alpha}
,\qquad
s_n^2 \log n = \alpha^{-2} \rho^{-2/\alpha} (\log n)^{-1+2/\alpha},
\]
the sequence $s_n \log n$ will tend to infinity for every choice of $\alpha>0$. However, $s_n^2 \log n$ tends to zero if and only if $\alpha >2$. This explains our parameter choice in  \refthm{weight_a_thm}.

Our proofs rely on first and second moment methods. The idea is that in the considered regime, as soon as there is one path of low weight, there will be many. In fact, we show in \refsect{main_proofs_sec} that there are sequences $(w_n)$ approximating the weight $W_n$ and $(k_n)$ approximating the hopcount $H_n$, such that the expected number of paths with weight at most $w_n$ and length $k_n$ tends to infinity. This large number of attractive paths is reliable in the sense that its variance is of lower order than its squared mean.
To estimate the required moments for these methods, we develop a path counting technique which captures numerous important characteristics of pairs of paths. These techniques may 
be of independent interest.

The second main idea is that an optimal path will contain edges which are all of approximately the same weight. In the spirit of the heuristics above, $\mathbb{E}[\min_{e \sim 1} X_e ]$ is of order $u_n$. Indeed, we find that the optimal path uses only edges with weight of order $u_n$. The hopcount can be approximated by $k_n = \lfloor s_n \log n \rfloor$ in agreement with the case $s_n=s>0$ constant. Therefore, the weight of the optimal path is of order $u_ns_n \log n$. This intuition is the basis 
for our estimates of the distribution function of $w(\p)$.

Like for the hopcount, our results for the weight $W_n$ are also consistent with the case where $s$ is a positive constant. By (\ref{weight_s_const}), when $s$ is constant $W_n / (u_n s_n \log n)$ converges to $s^{-1}\Gamma(1+1/s)^{-s}$ for $n\to\infty$, and by Stirling's formula $s^{-1}\Gamma(1+1/s)^{-s} \sim e$ for $s \downarrow 0$, in agreement with 
\refthm{weight_thm}. It would be of interest to understand also the fluctuations for the case $s_n \to 0$ 
and to compare these with the results of \cite{BhaHof12}.

It is fascinating that $n$-independent and $n$-dependent edge weight distributions, which at first sight have little in common, fall into the \emph{same} universality class. Surprisingly, first passage percolation problems in the same class do not have to have the same scaling for $W_n$: compare \refthm{weight_thm} and \refthm{weight_a_thm}. Indeed, as we shall observe below, the weight of the optimal path can even converge to zero for one distribution, but not for another in the same class. We conclude that the detailed asymptotics of $W_n$ are not an invariant of a universality class.  Nevertheless, we shall see that problems in the same class can be analyzed with similar methods.

A natural question is how the analyzed quantities behave in the different regimes. We will now discuss the current knowledge and conjectures.

{\bf{(i) Very small $\mathbf{s_n}$:}} When $s_n \log n$ does not tend to infinity but to a finite value $\gamma$, we call $s_n$ very small. In this regime $W_n/u_n$ is asymptotically bounded from above and below by finite, positive constants and the hopcount is tight. This was proven in \refthm{very_small_sn_thm} for $X_{e}\overset{d}{=}Z^{s_n}$, since $u_n \sim (\lambda n)^{-s_n} \to e^{-\gamma}$. To give a further example for an edge weight distribution in this regime, consider $X_e\overset{d}{=} E^{-\gamma}$ for $\gamma>0$ and $E$ an exponential random variable with mean one. In this case, $F(x)=\exp(-x^{-1/\gamma})$ and $u_n=(\log n)^{-\gamma}$. According to (\ref{defsn}), $\Phi(x)=e^{-x/\gamma}$ implies
\[
s_n= \frac{\gamma}{\log n}
\]
and $s_n \log n=\gamma$ is constant. This family of distributions was analyzed in great detail in \cite{BhaHofHoo12pre}. Bhamidi et al.\ show that the hopcount is tight and converges in distribution to a random variable with values in $\{\lfloor \gamma+1\rfloor, \lceil \gamma+1\rceil\}$. For the weight, it is shown that $W_n/u_n$ converges in probability to a finite, positive constant. The authors identify also the second order asymptotics, showing that a linear transformation of the weight converges weakly to a Gumbel distribution. Their methods apply under suitable assumptions also to the very small $s_n$ case, which is exploited in the proof of \refthm{very_small_2nd_order}; see \refsect{very_small_sn}.
Notice that, for edge weights $X_e\overset{d}{=} E^{-\gamma}$, the sequence $(u_n)_n$, and therefore also $W_n$, converge to zero, whereas in the very small $s_n$ case discussed in \refthm{very_small_sn_thm} $u_n$ converges to unity and $W_n$ is bounded away from zero.

{\bf{(ii)} $\mathbf{s_n} \boldsymbol{\log} \, \mathbf{n\to \infty}$ and $\mathbf{s_n^2} \boldsymbol{\log} \, \mathbf{n\to 0}$\bf{:}} The current article is concerned with the first order asymptotics. For the second order, we expect that there exists a sequence $k_n\approx s_n \log n$ such that $H_n=k_n$ whp.

{\bf{(iii)} $\mathbf{s_n}  \mathbf{\to 0}$ and $\mathbf{s_n^2} \boldsymbol{\log} \, \mathbf{n\to \infty}$\bf{:}} The case that $s_n$ is a null sequence but $s_n^2 \log n$ converges to infinity (or to a finite value) was not discussed yet. For the distribution $X_e\overset{d}{=}e^{-(E/\rho)^{1/\alpha}}$ the corresponding parameter choice is $\alpha \in (1,2)$. Our results for the lower bound on the weight of the optimal path are valid in this regime. Since the first order asymptotics in the cases $s_n^2 \log n \to 0$ and $s$ constant agree, we expect these to remain valid. Moreover, we believe that the hopcount satisfies a central limit theorem. Results for this case would be of great interest.

{\bf{(iv)} $\mathbf{s >0}$\bf{ constant:}} The case of a fixed $s$ is treated in \cite{BhaHof12}, see the discussion around (\ref{weight_s_const}).

{\bf{(v)} $\mathbf{s_n \to \infty}$\bf{:}} The opposite extreme to $(s_n)$ being a null sequence is when $s_n$ tends to infinity. This case is discussed in a separate article \cite{EGHN12bPre}. The speed of $(s_n)$ plays again an important role and one expects different behaviours when $s_n\ll n^{1/3}$, $s_n \approx n^{1/3}$ or $s_n \gg n^{1/3}$.

\subsection{Outline}

The outline of this article is as follows. We start by presenting the path counting technique in \refsect{path_counting_sec}. This technique is used to establish first and second moment methods for the scaling of the weight of the smallest weight path in a general setup in \refsect{fst_snd_mm_sec}. To apply these methods, estimates on the distribution function of $w(\p)$ are required. A general method how to estimate distribution functions of this type is presented in \refsect{gen_df_sec}. In \refsect{sn_df_sec} this method is made explicit for the case that $X_e \overset{d}{=} Z^{s_n}$ and, for not too long paths, the precise scaling of the distribution function close to zero is derived. Then all ingredients are collected and we prove Theorems \ref{weight_thm}, \ref{very_small_sn_thm} and \ref{very_small_2nd_order} in \refsect{main_proofs_sec} and \refthm{weight_a_thm} in \refsect{alpha_sec}.

\section{Path counting}\label{path_counting_sec}
To estimate the weight of the optimal path and the number of edges it uses, it is of interest how many paths of a certain length and weight there are. The set of paths between vertex $1$ and $2$ is denoted by
\begin{align*}
\S_{1,2}&:=\{\p:\, \p \text{ self-avoiding path between vertex $1$ and $2$ in $K_n$}\},\\
\S_{1,2}(k)&:= \{\p \in \S_{1,2}: \,\p\text{ uses exactly $k$ edges}\} \qquad \qquad \qquad\forall k \in \N.
\end{align*}
If two paths do not have any common edges, then their weights are independent. When they do overlap, their weights are more strongly correlated the more edges they have in common. Hence, we will estimate the number of pairs of paths in $\mathcal{S}_{1,2}(k)$ that share $l$ edges, $l\in [k]$. It is important for our counting that the paths are self-avoiding. We start with some terminology.

\begin{definition}{\bf{(Excursions and gaps)}}
Let $\p,\q \in \S_{1,2}$.
Every connected subpath of $\q$ which contains only edges in $\q \setminus \p$, starts and ends with a vertex in $\p$, and has no other vertices in common with $\p$ is called an \emph{excursion} of $\q$ from $\p$.
To order excursions, the path $\q$ is considered as a directed path with starting point $1$ and endpoint $2$. The excursions of $\p$ from $\q$ are called \emph{gaps}.
\end{definition}

The counting method developed in this section allows us to estimates the number of pairs of paths that satisfy numerous conditions on their shape. The results obtained are more detailed than necessary for the situation of this article and the proof of \refthm{path_counting_thm}. We chose to include the detailed estimates because they are of independent interest and can be used to generalize the methods of this paper to other graphs. For instance, this was done to prove an upper bound for the hopcount in the context of inhomogeneous random graphs. The procedure explained below was altered slightly such that an intersection of two paths without a common edge does not trigger a new excursion (see Section 9.4 and in particular Lemma 9.18 in \cite{Hof12}).

\begin{theorem}\label{path_counting_thm}{\bf{(The number of pairs of paths)}}
If $k=o(n^{1/3})$, then
\[
\mathcal{P}(k,l):=\{(\p,\q)\in \mathcal{S}_{1,2}(k)^2 : |\p \cap \q|=l \}
\]
satisfies
\[
|\mathcal{P}(k,l)|=\left\{\begin{array}{ll}
(l+1)n^{2k-l-2}(1+O(k^4/n)) & \;\text{if }l\in [k-2],\\
0 & \;\text{if }l=k-1,\\
n^{k-1} (1+O(k^2/n)) & \;\text{if }l=k,\end{array}\right.
\]
where the error term in the first case holds uniformly in $l$.
\end{theorem}

\begin{proof} Our counting idea is the following: We think of one path $\p$ being chosen explicitly first. For the further counting we fix the number $m$ of excursions the second path $\q$ is making from $\p$. Now the common vertices and edges are chosen. The path $\q$ is constructed from these by first ordering the common pieces, then directing them and choosing the length of the excursions. In the end the actual paths of the excursions are determined.

This procedure is now made precise: the path $\p$ has $k$ edges and, therefore, $k+1$ vertices. Vertex $1$ and $2$ are given. Hence, a sequence of $k-1$ vertices is chosen from $n-2$ vertices giving us $\prod_{j=2}^k (n-j)$ possible choices. In case $l=k$, this already proves the claim since Stirling's estimates imply for $k = o(n^{1/2})$
\begin{equation}\label{size_S12k_eq}
|\mathcal{P}(k,k)|=|\mathcal{S}_{1,2}(k)|= \prod_{j=2}^k (n-j)= \frac{(n-2)!}{(n-k-1)!}=n^{k-1}(1+O(k^2/n)).
\end{equation}

Two self-avoiding paths with $k$ edges between vertex $1$ and $2$ cannot have exactly $k-1$ common edges. Hence, we can assume $l\le k-2$ and $k\ge 3$ from now on.

We determine the number of possible choices for common vertices and edges when $q$ makes $m$ excursions from $p$. One example is displayed in Figure 1.

\begin{center}
\begin{tikzpicture}

\draw [thick] (0,0) -- (8,0);
\draw [green!80!black, thick] [yshift= 2pt] (0,0) -- (1,0);
\draw [green!80!black, thick] [yshift= 2pt] (2,0) -- (4,0);

\foreach \x in {0,...,8}
\filldraw (\x,0) circle (0.08cm);

\foreach \x in {0,...,4}
\draw [green!80!black, thick] (\x,0) circle (0.08cm);

\draw [green!80!black, thick] (6,0) circle (0.08cm);
\draw [green!80!black, thick] (8,0) circle (0.08cm);

\end{tikzpicture} \parskip 10pt

{\small {\bf Figure 1:} A possible choice of common vertices and edges for $k=8$, $l=3$ and $m=3$. Here, $x_1=1$, $x_2=2$, $x_3=x_4=0$ and $r_1=1$, $r_2=r_3=2$.}
\end{center}

Denote by $x_j$ the number of common edges before gap $j$, where $j \in [m]$, and by $x_{m+1}$ the number of common edges after the $m$th gap. There are
\[
|\{ (x_1,\dotsc,x_{m+1}) \in \mathbb{N}_0^{m+1}: \sum_{j=1}^{m+1} x_j=l\}| = \binom{m+l}{m}
\]
possible choices of these parameters. When $r_j$ denotes the length of gap $j$ (in edges), then
\[
|\{(r_1,\dotsc,r_m) \in \N^m : \sum_{j=1}^m r_j = k-l\}| = |\{(\hat{r}_1,\dotsc,\hat{r}_m) \in \mathbb{N}_0^m : \sum_{j=1}^m \hat{r}_j= k-l-m\}|=\binom{k-l-1}{m-1}
\]
is the number of possible gap lengths. Here we have taken into account that a gap needs to have at least length $1$. In total there are $\binom{m+l}{m} \binom{k-l-1}{m-1}$ possible choices for the common edges and vertices.

In the next step we determine the order of the $m+1$ subpaths in $\p \cap \q$. The first one contains vertex $1$, the last one contains vertex $2$. Hence, we have $(m-1)!$ choices to order the remaining. If one of these parts contains an edge, there are two possible choices in which direction it can be used. See Figure 2 for an example.

\begin{center}
\begin{tikzpicture}

\foreach \x in {2,5}
\draw [green!80!black, thick] (\x-1,0) -- (\x,1) -- (\x+1,0);

\draw [green!80!black, thick] (1,3) -- (2.5,2) -- (4,3);
\draw [green!80!black, thick] (3,3) -- (4.5,4) -- (6,3);

\foreach \y in {0,3}
{
\draw [thick] (0,\y) -- (7,\y);
\foreach \x in {0,3,6}
{
\draw [green!80!black, thick] [yshift= 2pt] (\x,\y) -- (\x+1,\y);
}
\foreach \x in {0,...,7}
\filldraw (\x,\y) circle (0.08cm);
\foreach \x in {0,1,3,4,6,7}
\draw [green!80!black, thick] (\x,\y) circle (0.08cm);
}

\foreach \x in {2,5}
\filldraw [green!80!black, thick] (\x,1) circle (0.08cm);

\filldraw [green!80!black, thick] (2.5,2) circle (0.08cm);
\filldraw [green!80!black, thick] (4.5,4) circle (0.08cm);

\end{tikzpicture} \parskip 10pt

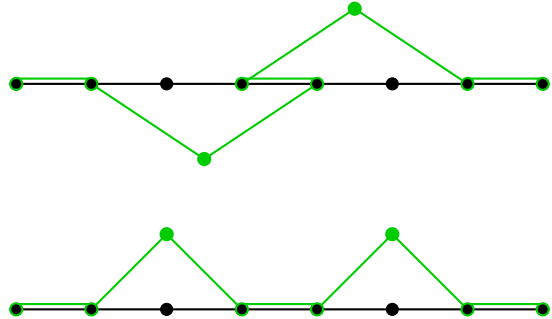
{\small {\bf Figure 2:} The black path is $\p$, the green one is $\q$. We see two different ways of orienting the middle piece, even though there is no choice for ordering the common pieces. $k=7$, $l=3$ and $m=2$.}
\end{center}

Hence, the number of possible directions equals $2^{\sum_{j=2}^m \mathbbm{1}_{x_j \ge 1}} \le 2^{m-1}$. When $m=1$, this inequality is an identity.

Now the lengths of the excursions are determined. If a gap has only one edge and the next common vertex is the other side of this edge, then the excursion needs to have at least $2$ edges. Otherwise an excursion may have only one edge (see Figure 3 or Figure 4 for an example). Hence, we can bound the number of possible choices for the lengths of the excursions by
\[
\big|\{(t_1,\dotsc,t_m) \in \N^m : \sum_{j=1}^m t_j = k-l\}\big|=\binom{k-l-1}{m-1}.
\]
When $m=1$ the excursion has length $k-l$ giving us no choice. Hence, the estimate is precise in the case $m=1$. In Figure 3 the used variables are visualized.

\begin{center}
\begin{tikzpicture}

\tikzstyle{every node}=[font=\small]

\foreach \x in {0,4,6}
\draw [green!80!black, thick] [yshift= 2pt] (\x,0) -- (\x+1,0);

\foreach \x in {1,4}
\draw [green!80!black, thick] (\x,0) .. controls (\x+.5,-1) and (\x+1.5,-1) .. (\x+2,0);

\draw [green!80!black, thick] (3,0) -- (4,1) -- (5,0);
\filldraw [green!80!black, thick] (4,1) circle (0.08cm);

\draw [thick] (0,0) -- (7,0);

\foreach \x in {0,...,7}
\filldraw (\x,0) circle (0.08cm);
\foreach \x in {0,1,3,4,5,6,7}
\draw [green!80!black, thick] (\x,0) circle (0.08cm);

\draw (1.2,0.4) node{$x_1=1$};
\draw (3.5,-0.4) node{$x_2=0$};
\draw(5.3,0.4) node{$x_3=1$};
\draw (6.7,0.4) node{$x_4=1$};
\draw (2,-1) node{$t_1=1$};
\draw (4,1.4) node{$t_2=2$};
\draw (5,-1) node{$t_3=1$};
\draw (2,1) node{$r_1=2$};
\draw (3.5,-1) node{$r_2=1$};
\draw (5.3,1) node{$r_3=1$};

\end{tikzpicture}\parskip 10pt

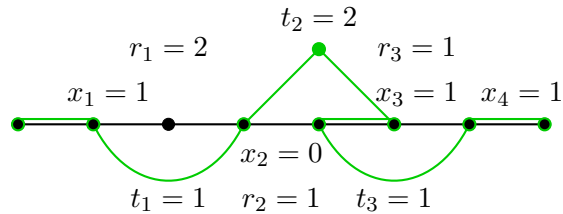
{\small {\bf Figure 3:} Visualization of used variables. $k=7$, $l=3$ and $m=3$.}
\end{center}

Finally, the explicit shape of the excursions is determined by choosing their sequence of vertices. Path $\q$ has $k+1$ vertices. Exactly $l+m+1$ vertices are in common with $\p$: one at the end of each common edge, one is vertex $1$ itself and one at the end of each excursion. Moreover, $k+1$ vertices were already used by $\p$. Hence, a sequence of $k+1 - (l+m+1)$ vertices needs to be chosen from $n-(k+1)$ vertices, leaving us with
\[
\prod_{j=0}^{k+1 - (l+m+1) -1} (n- (k+1) -j)= \prod_{j=k+1}^{2k-l-m}(n-j)
\]
possible choices. As each excursion uses at least one edge, $m$ is at most $k-l$. Figure 4 shows that this bound cannot be improved.

\begin{center}
\begin{tikzpicture}

\draw [thick] (0,0) -- (5,0);
\draw [green!80!black, thick] [yshift= 2pt] (1,0) -- (2,0);
\draw [green!80!black, thick] [yshift= 2pt] (3,0) -- (4,0);

\draw [green!80!black, thick] (0,0) .. controls (1,1) and (2,1) .. (3,0);
\draw [green!80!black, thick] (1,0) .. controls (2,-1) and (3,-1) .. (4,0);
\draw [green!80!black, thick] (2,0) .. controls (3,1) and (4,1) .. (5,0);

\foreach \x in {0,...,5}
\filldraw (\x,0) circle (0.08cm);
\foreach \x in {0,...,5}
\draw [green!80!black, thick] (\x,0) circle (0.08cm);

\end{tikzpicture}

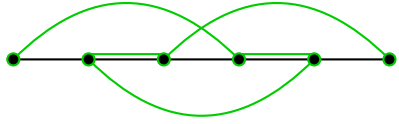
{\small {\bf Figure 4:} Excursions with length one and the maximal choice of $m$.\\
$k=5$, $l=2$ and $m=3$.}
\end{center}

We obtain for $l\in [k-2]$
\[
|\mathcal{P}(k,l)|\le \sum_{m=1}^{k-l} \binom{m+l}{m} \binom{k-l-1}{m-1}^2 (m-1)! 2^{m-1} \prod_{j=2}^{2k-l-m} (n-j).
\]
By Stirling's Formula $\prod_{j=2}^{2k-l-1} (n-j)=n^{2k-l-2}(1+O(k^2/n))$. The coefficient for $m=1$ equals $l+1$ and was determined precisely. Using $n-j \le n$, it remains to show that
\[
\sum_{m=2}^{k-l} \binom{m+l}{m} \binom{k-l-1}{m-1}^2 (m-1)! 2^{m-1} n^{-m+1}=O(k^4/n).
\]
To this end, we estimate
\begin{align*}
\binom{m+l}{m} &= \frac{(l+m)\cdots(l+1)}{m!}\le \frac{k^m}{m!} \quad \qquad \qquad \forall \, m \le k-l,\\
\binom{k-l-1}{m-1}^2 (m-1)! &= \left(\frac{(k-l-1)!}{(k-l-m)!}\right)^2 \frac{1}{(m-1)!} \le \frac{k^{2(m-1)}}{(m-1)!}.
\end{align*}
In particular, $k=o(n^{1/3})$ implies
\[
\sum_{m=2}^{k-l} \binom{m+l}{m} \binom{k-l-1}{m-1}^2 (m-1)! 2^{m-1} n^{-m+1} \le k \sum_{m=2}^{\infty} \left(\frac{2k^3}{n}\right)^{m-1} = k \frac{\frac{2k^3}{n}}{1-\frac{2k^3}{n}}=O(k^4/n).\qedhere
\]
\end{proof}

\section{\texorpdfstring{First and second moment methods for $W_n$}{1st and 2nd moment methods for Wn}}\label{fst_snd_mm_sec}

This section is devoted to the introduction of a general method to obtain first order estimates for the weight of the shortest path $W_n$. We start with some notation.

For $b \ge 0$ and $k \in \N$, the number of paths between vertex $1$ and $2$ with weight at most $b$ using exactly $k$ edges is denoted by
\[
N_k(b):= \sum_{\p \in \mathcal{S}_{1,2}(k)} \mathbbm{1}_{w(\p)\le b}.
\]
We remark that $\S_{1,2}, \S_{1,2}(k)$ and $N_k(x)$ depend on $n$. This is suppressed in the notation since it will always be clear which graph is considered. For the distribution function of the weight of a path $\p \in \S_{1,2}(k)$ we write
\[
F_n^{*k}(x):=\P\Big(\sum\limits_{e \in \p} X_e\le x\Big) \qquad \forall \, x\in \R.
\]

\begin{proposition}\label{lower_gen_prop}{\bf{(Lower bound for $W_n$)}} Let $(d_n)_{n\in \mathbb{N}}\in (0,\infty)^{\N}$ such that
\[
\sum_{k=1}^{n-1} n^{k-1}F_n^{*k}(d_n) \to 0 \qquad \text{for } n\to \infty.
\]
Then
\[
W_n \ge d_n \qquad \text{with high probability}.
\]
\end{proposition}

\begin{proof}
Note that $W_n <d_n$ implies that there exists a $k \in \N$ such that $N_k(d_n)\ge 1$. Markov's inequality and (\ref{size_S12k_eq}) now yield
\[
\P(W_n < d_n) \le \sum_{k=1}^{n-1} \E[N_k(d_n)] = \sum_{k=1}^{n-1} \prod_{j=2}^k (n-j) F_n^{*k}(d_n) \le \sum_{k=1}^{n-1} n^{k-1}F_n^{*k}(d_n)=o(1).\qedhere
\]
\end{proof}

The method for the upper bound relies on the idea that the variance of attractive paths is small compared to its squared mean. Or, more intuitively, with high probability there is a reliably large number of attractive paths.

\begin{proposition}\label{gen_upper_prop}{\bf{(Upper bound for $W_n$)}} Let $(b_n)_{n\in \mathbb{N}}\in (0,\infty)^{\N}$. When there exists a sequence $(k_n)_{n\in \mathbb{N}}\in \N^{\N}$ such that $k_n=o(n^{1/4})$ and
\begin{align*}
\sum_{l=1}^{k_n-2} (l+1)n^{-l} \frac{F^{*(k_n-l)}_n (b_n)}{F^{*k_n}_n (b_n)}\to 0& \qquad \text{for }n\to \infty, \\
\E[N_{k_n}(b_n)] \to \infty& \qquad \text{for } n\to \infty,
\end{align*}
then
\[
W_n \le b_n \qquad \text{with high probability}.
\]
\end{proposition}

\begin{proof}
When $W_n >b_n$, there is no path of weight less or equal than $b_n$. This fact and Chebychev's inequality imply that for any $k_n \in \N$
\begin{equation}\label{markov_estimate}
\P(W_n > b_n) \le \P(N_{k_n}(b_n) =0) \le \frac{\text{Var}(N_{k_n}(b_n))}{\E[N_{k_n}(b_n)]^2}.
\end{equation}
To show that the right-hand side converges to zero for the choice $k_n$ from the statement, we estimate the moments of $N_{k_n}(b_n)$. \refthm{path_counting_thm} implies
\begin{equation}\label{mean_estimate}
\E[N_{k_n}(b_n)] = \sum_{\p \in \mathcal{S}_{1,2}(k_n)} \P(w(\p)\le b_n) = |\mathcal{S}_{1,2}(k_n)| F^{* k_n}_n (b_n)= (1+o(1)) n^{k_n-1}F^{*k_n}_n (b_n).
\end{equation}
For the variance, we have
\begin{align*}
\text{Var}(N_{k_n}(b_n))&=\E\Big[\big(\sum_{\p \in \S_{1,2}(k_n)} \mathbbm{1}_{w(\p)\le b_n}-\P(w(\p)\le b_n)\big)^2\Big]\\
&=
\sum_{\p,\q \in \S_{1,2}(k_n)} \big[\P(w(\p)\le b_n,w(\q)\le b_n)-\P(w(\p)\le b_n)\P(w(\q)\le b_n) \big]\\
&\le
\sum_{\p,\q \in \S_{1,2}(k_n), \p \cap \q \not= \emptyset} \P(w(\p)\le b_n, \, w(\q) \le b_n)
\\
&=
\sum_{l=1}^{k_n} \sum_{(\p,\q) \in \mathcal{P}(k_n,l)} \P(w(\p)\le b_n, \, w(\q) \le b_n).
\end{align*}
We decompose the weights $w(\p)$ and $w(\q)$ into the part due to the common edges and the part originating from distinct subpaths. When $(Y_{i})_{i\in \N}, (Y_{i}')_{i\in \N}$, $(Y_i'')_{i\in \N}$ are independent random variables and $Y_{i}$, $Y_i'$ and $Y_{i}''$ have distribution function $F^{*i}_n$, \refthm{path_counting_thm} and nonnegativity of $Y_{l}$ imply
\begin{align*}
\text{Var}&(N_{k_n}(b_n))\\
&\le
 |\mathcal{S}_{1,2}(k_n)|F^{*k_n}_n(b_n)+(1+o(1)) \sum_{l=1}^{k_n-2} (l+1)n^{2k_n-l-2} \P(Y_{l}+Y_{k_n-l}'\le b_n, Y_{l}+Y_{k_n-l}''\le b_n)\\
&\le |\mathcal{S}_{1,2}(k_n)|F^{*k_n}_n (b_n)+(1+o(1)) \sum_{l=1}^{k_n-2} (l+1)n^{2k_n-l-2} F^{*k_n}_n (b_n)F^{*(k_n-l)}_n(b_n).
\end{align*}
Combining the last estimate with (\ref{markov_estimate}) and (\ref{mean_estimate}), we obtain
\[
\P(W_n > b_n)\le \E[N_{k_n}(b_n)]^{-1}+ (1+o(1)) \sum_{l=1}^{k_n-2} (l+1)n^{-l} \frac{F^{*(k_n-l)}_n (b_n)}{F^{*k_n}_n (b_n)}=o(1).\qedhere
\]
\end{proof}

Notice that $n^{-l}$ is the dominating term in the variance estimate and for $l=1$ we neglect only one summand. Therefore, we expect the approximation of the variance to be quite reliable. The right choice of sequence $k_n$ is crucial to make the method work. Intuitively, the sequence $k_n$ may be thought of as an approximation of the hopcount. Good guesses can be obtained using the heuristics explained in \refsect{discussion}.

\section{\texorpdfstring{The distribution of $w(\p)$}{Distribution of w(p)}}
In this section estimates for the distribution function of the weight of a path $\p \in S_{1,2}(k)$ are derived. These are needed for the application of the methods developed in \refsect{fst_snd_mm_sec}.

\subsection{Estimates for the distribution function}\label{gen_df_sec}
To apply the derived method, we rely on estimates on the distribution function of $w(\p)$ for fixed $n$. Hence for this section, let $F$ be a distribution function which is concentrated on $(0,\infty)$ and absolutely continuous with Lebesgue density $f$. For $(X_i)_{i \in \mathbb{N}}$ i.i.d.\ with distribution function $F$, we write for every $k\in \mathbb{N}$
\[
F^{*k}(x):= \P\left(\sum_{i=1}^k X_i \le x\right), \qquad x \in \mathbb{R}.
\]
\begin{proposition}[{\bf{Rough bounds on $F^{*k}$}}]\label{gen_rough_prop}
For all $k\in \mathbb{N}$ and $x\in [0, \infty)$,
\[
F(x/k)^k \le F^{*k}(x)\le \frac{x^k}{k!} \max_{\sum_{i=1}^k x_i \le x} \left(\prod_{i=1}^k f(x_i) \right).
\]
\end{proposition}

\begin{proof}
To derive the lower bound, we estimate for all $k \in \mathbb{N}$ and $x \in [0,\infty)$
\[
F^{*k}(x)=\P\left(\sum_{i=1}^k X_i \le x\right) \ge \P\left(\cap_{i=1}^k \{X_i \le x/k\}\right) = F(x/k)^k.
\]
The upper bound is obtained by
\begin{align*}
F^{*k}(x)&=\P\left(\sum_{i=1}^k X_i \le x\right)= \int_{\mathbb{R}^k} \mathbbm{1}_{\sum_{i=1}^k y_i \le x} \prod_{i=1}^k f(y_i)\, d(y_1,\dotsc,y_k)\\
&\le \max_{\sum_{i=1}^k x_i \le x } \left( \prod_{i=1}^k f(x_i) \right) \int_{[0,\infty)^k} \mathbbm{1}_{\sum_{i=1}^k y_i \le x}\, d(y_1,\dotsc,y_k) \\
&= \max_{\sum_{i=1}^k x_i \le x} \left( \prod_{i=1}^k f(x_i) \right) \frac{x^k}{k!}.\qedhere
\end{align*}
\end{proof}

To compute the value of the maximum in the upper bound, the following lemma is useful for many distributions. For a differentiable function $h$ we call $x$ stationary, if $x$ is in the domain of $h$ and $h'(x)=0$.

\begin{lemma}\label{strict_convex_lem}
Let $k \in \N$, $\bar{x}\in (0,\infty]$ and $h\in C^1((0,\bar{x}))$ strictly convex with $\lim_{x \downarrow 0}h(x)=\infty$. Let $x^*\in (0,\bar{x})$ be the unique stationary point of $h$ if such a point exists or $x^*=\bar{x}$ otherwise.\\
If $f(x)= \exp(-h(x))$ for all $x\in (0,\bar{x})$ and $f(x)=0$ for all $x\le 0$, then
\[
\max_{\sum_{i=1}^k x_i \le x} \left( \prod_{i=1}^k f(x_i) \right)= f(x/k)^k
\]
for all $x \le kx^*$ with $0<x<\bar{x}$.
\end{lemma}

\begin{proof} First, notice that the optimization problem is solvable since the feasible region is compact and the involved functions are sufficiently smooth. Because $x < \bar{x}$, the constraint $\sum_{i=1}^kx_i\le x$ implies $x_i< \bar{x}$ and we obtain
\[
\max_{\sum_{i=1}^k x_i \le x} \left( \prod_{i=1}^k f(x_i) \right)= \exp\big(-\min_{\sum_{i=1}^k x_i \le x} \sum_{i=1}^k h(x_i)\big).
\]
The convexity of $h$ implies convexity of $(x_1,\dotsc,x_k) \mapsto \sum_{i=1}^k h(x_i)$. Since the condition $\sum_{i=1}^k x_i \le x$ is linear, the optimization problem is convex. A necessary and sufficient condition for a solution are the KKT-conditions (see for example \cite[Section 5.5.3]{BV04}), i.e.\ the existence of $\lambda \in [0,\infty)$ such that
\[
h'(x_i)=-\lambda \le 0 \qquad \forall \, i \in [k], \qquad \sum_{i=1}^k x_i \le x, \qquad \lambda\big(\sum_{i=1}^k x_i - x\big)=0.
\]
Strict convexity of $h$ implies that $h'$ is one-to-one. Hence, $x_1=x_i$ for all $i \in [k]$ whenever $(x_1,\dotsc,x_k)$ is a solution. If $h$ has a stationary point $x^*\in [0,\bar{x}]$, then this point is unique and positive since $\lim_{x\downarrow 0}h(x)=\infty$. In case of a stationary point with $kx^*\le x$ the solution is given by $x_i=x^*$ for all $i\in [k]$. If $kx^*> x$ or if there is no stationary point, then $\lambda\not=0$ and $\sum_{i=1}^k x_i = x$ implies $x_i=x/k$, $i=1,\dotsc,k$.
\end{proof}

\subsection{\texorpdfstring{Distribution of $w(\p)$ for the $X_e\stackrel{d}{=}Z^{s_n}$ class}{Distribution of w(p) for the X=Zsn class}}\label{sn_df_sec}

In this section, we present estimates for the distribution function $F^{*k}$ in the case $X_1\overset{d}{=} Z^{s}$, where $s\in (0,1)$. Here $Z$ denotes a positive random variable with distribution function $G$. Write $p:=1/s$.

We start with the rough estimates from \refsect{gen_df_sec} before we derive the precise scaling for $s=s_n \to 0$ and moderate $k$. The results of this section will enable us to present the full scope of the techniques explained in the previous sections.

\begin{lemma}\label{rough_sn_lem}
Assume there exists a $\hat{x}\in (0,\infty)$ such that $G\in C^1([0,\hat{x}])$ and $G'(0+)>0$. Then there exists a constant $c>0$ such that, for all $p=1/s>1$, $k\in \mathbb{N}$ and $x^p\in [0, 1 \land \hat{x}]$,
\[
G\big((x/k)^p\big)^k \le F^{*k}(x)\le (cp)^k \left(\frac{x}{k}\right)^{kp}.
\]
\end{lemma}

\begin{proof}
We have $F(x)=\P(Z^{s}\le x)=G(x^{p})$. Thus, \refprop{gen_rough_prop} yields the lower bound. For the upper bound, we compute the Lebesgue density
\[
f(x)=\frac{dF}{dx}(x)= px^{p-1}G'(x^p).
\]
Denote $d=\max\{G'(x): x \le 1 \land \hat{x}\}$. Then $g\in (0,\infty)$ by the assumptions on $G$ and $f(x) \le dp x^{p-1}$ for all $x \le 1 \land \hat{x}$. \refprop{gen_rough_prop}, an application of \reflemma{strict_convex_lem} to $h(x):= -(p-1)\log x$ with $\bar{x}=x^*=\infty$, and Stirling's estimates yield for $x \in [0, 1 \land \hat{x}]$
\[
F^{*k}(x)\le \frac{x^k}{k!}(dp)^k \Big(\frac{x}{k}\Big)^{(p-1)k} \le (edp)^k \Big(\frac{x}{k}\Big)^{kp}.\qedhere
\]
\end{proof}

The simple structure of the distribution function in this example allows us to give a very precise estimate how the distribution function behaves close to zero. We denote $p_n:=1/s_n$ for all $n \in \N$.

\begin{lemma}\label{asymp_sn_lem}
Let $\hat{x} \in (0,\infty)$ such that $G \in C^2([0,\hat{x}])$, $G'(0+)=1$. Let $X_e\overset{d}{=}Z^{s_n}$ with $s_n\to 0$. If $(k_n)_{n \in \N}\in \N^{\N}$ and $(b_n)_{n \in \N}\in (0,\infty)^{\N}$ satisfy $k_n=o(1/s_n)$ and $b_n=o(1)$, then
\[
F_n^{*k}(b_n)=\left(\frac{b_n}{k}\right)^{kp_n} p_n^{(k-1)/2} a_k (1+o(1)),
\]
where the error is uniform in $k \in [k_n]$ and $a_k$ is defined in (\ref{defak}).
\end{lemma}

\begin{proof}
By Stirling's formula,
\begin{equation} \label{stirlingForFk}
\frac{\Gamma(p_n+1)^k}{\Gamma(kp_n +1)}=\frac{1}{k^{kp_n}}\frac{(2 \pi)^{k/2}}{\sqrt{2 \pi k}} p_n^{(k-1)/2} (1+O(k/p_n)) \qquad \text{for }n\to \infty.
\end{equation}
Since $|k/p_n| \le k_n/p_n=o(1)$, the error bound holds uniformly in $k\in [k_n]$. Hence, it is sufficient to prove that there exists $p_0 >0$ and $x_0>0$ such that for all $k\in \N$,
\begin{equation}\label{detailed_df_sn}
F^{*k}(x)=x^{kp} \frac{\Gamma(p+1)^k}{\Gamma(kp+1)} (1+r_k(p,x) x^p) \qquad \forall p \ge p_0, x \in [0,x_0],
\end{equation}
where $|r_k(p,x)| \le \bar r_k(p):=4^p (\frac{k^k}{(k+1)^{k+1}})^p 2^{k}D$ for all $x \in [0,x_0]$ and $D=\max\{|G''(x)|: x \in [0,1 \land \hat{x}]\}<\infty$. Indeed, one easily checks that $k_n=o(p_n)$ implies that $\bar r_k(p_n)$ is decreasing in $k\in [k_n]$ for sufficiently large $n$. The error is thus bounded by $\bar r_1(p_n)b_n^{p_n}=2Db_n^{p_n}$.

Note that if $Z$ is uniform on $(0,1)$, so that $F(x)=x$ for $x\leq x_0=1$, then (\ref{detailed_df_sn}) holds with no error term. For the general case, we will prove (\ref{detailed_df_sn}) by induction on $k$ by suitably controlling the error terms.

Let $x_0=\hat{x}\land 1$. By a Taylor expansion, there exist functions $c_p$ and $d_p$ such that, for all $x \le x_0$,
\[
\begin{cases}
G(x^p)=x^p+\tfrac{1}{2}c_p(x)x^{2p},\\
G'(x^p)=1+d_p(x)x^p, \end{cases} \qquad \text{with }|c_p(x)|,|d_p(x)|\le D.
\]
In particular,
\[
F(x)=G(x^p)=x^p(1+\tfrac{1}{2}c_p(z)x^{p}),
\]
and $|\frac{1}{2}c_p(x)|\le \frac{1}{2}D\le 2D=\bar r_1(p)$. The choice of $p_0$ will be explained in the induction step; for the base it is not needed.

Assume that the statement is proven for $F^{*l}$ with $l\le k$. Using the independence assumption, we obtain for $x \le x_0$
\begin{align*}
&F^{*(k+1)}(x)=\int_0^x F^{*k}(x-y) \,F(dy) = \int_0^1 F^{*k}(x(1-y))  F'(xy) x \; dy \\
&\quad= \int_0^1
\Big[
(x(1-y))^{kp} \frac{\Gamma(p+1)^k}{\Gamma(kp+1)} \Big(1+r_k(p,x(1-y))(x(1-y))^p\Big) p(xy)^{p-1} \big(1+d_p(xy)(xy)^p\big) x \Big]
\,dy \\
&\quad=\frac{\Gamma(p+1)^k}{\Gamma(kp+1)}p x^{(k+1)p}\int_0^1
\Big[
(1-y)^{kp}y^{p-1} \big(1+r_k(p,x(1-y))(x(1-y))^p\big)
\big(1+d_p(xy)(xy)^p\big)
\Big] \,
dy.
\end{align*}
Expanding the brackets, the integral without any error term equals $B(kp+1,p)$, where $B(\cdot,\cdot)$ denotes the beta function. The error term can be estimated by
\[
x^p\bar r_k(p) B((k+1)p+1,p)+x^pDB(kp+1,2p)+x^{2p}D\bar r_k(p) B((k+1)p+1,2p).
\]
Hence, it remains to prove that, for all $p\ge p_0$ and $x\in [0,x_0]$, the quantity
\begin{equation}\label{restterm_df}
\bar r_k(p)\frac{B((k+1)p+1,p)}{B(kp+1,p)}+D\frac{B(kp+1,2p)}{B(kp+1,p)}+x^pD \,\bar r_k(p)\frac{B((k+1)p+1,2p)}{B(kp+1,p)}
\end{equation}
is less than or equal to $\bar r_{k+1}(p)$. By Stirling's formula, for all $m_1,m_2\in \N$
\[
B(m_1p+1,m_2p) =\sqrt{\frac{2\pi}{p}}\left(\frac{m_1^{m_1}m_2^{m_2}}{(m_1+m_2)^{m_1+m_2}}\right)^p\sqrt{\frac{m_1}{m_2(m_1+m_2)}}(1+O(1/p)),
\]
where the error term is uniform in $m_1,m_2 \in \N$. Hence,
\begin{align*}
\frac{B((k+1)p+1,p)}{B(kp+1,p)}&=(1+O(1/p))\frac{k+1}{\sqrt{k(k+2)}}\frac{1}{2}\frac{\bar r_{k+1}(p)}{\bar r_k(p)}\\
\frac{B(kp+1,2p)}{B(kp+1,p)}&= (1+O(1/p)) \frac{\sqrt{k+1}}{\sqrt{2(k+2)}}\frac{1}{2^{k+1}D} \bar r_{k+1}(p)\\
\frac{B((k+1)p+1,2p)}{B(kp+1,p)}&= (1+O(1/p)) \Big(\frac{(k+2)^{k+2}}{(k+3)^{k+3}}\Big)^p\frac{k+1}{\sqrt{2k(k+3)}} \frac{4^p}{2} \frac{\bar r_{k+1}(p)}{\bar r_k(p)}.
\end{align*}
Choose $p_0$ such that $1+O(1/p) \le \frac{3}{2\sqrt{2}}$ for all $p\ge p_0$ and decrease $x_0$ such that $Dx^p \le 1/12$ for all $p \ge p_0$ and $x \le x_0$. Then the expression in (\ref{restterm_df}) can be estimated by
\begin{align*}
\tfrac{3}{2\sqrt{2}}(\tfrac{k+1}{\sqrt{k(k+2)}}\tfrac{1}{2}+\tfrac{\sqrt{k+1}}{\sqrt{2(k+2)}}\tfrac{1}{2^{k+1}}+\tfrac{1}{12} \tfrac{1}{4^{4p}} \tfrac{k+1}{\sqrt{2k(k+3)}}\tfrac{4^p}{2}&)\bar r_{k+1}(p)\\
&\le \tfrac{3}{2\sqrt{2}}(\tfrac{1}{\sqrt{2}}+\tfrac{1}{\sqrt{2}2^2}+\tfrac{1}{12} \tfrac{1}{\sqrt{2}})\bar r_{k+1}(p) =\bar r_{k+1}(p).\qedhere
\end{align*}
\end{proof}

\section{\texorpdfstring{Proofs for $X_e\stackrel{d}{=}Z^{s_n}$}{Proofs for X=Zsn}}\label{main_proofs_sec}

We now turn to the proofs of results for the network with edge weight distribution $X_e\overset{d}{=}Z^{s_n}$, where $Z$ denotes a positive random variable with distribution function $G$. In Sections~\ref{weight_opt_path_ss} and \ref{hop_sec} 
we assume that $G\in C^2([0,\hat{x}])$ for some $\hat{x}\in (0,\infty)$ and $\lambda=G'(0+)>0$. For the proofs in the very small $s_n$ regime, which are collected in \refsubsect{very_small_sn}, 
we will also work under slightly weaker conditions.

\subsection{The weight of the optimal path}\label{weight_opt_path_ss}
The statements (\ref{weight_bounds_eq}) and (\ref{Wn_in_prob_eq}) of \refthm{weight_thm} will follow from the upper and lower bounds stated in Theorems \ref{upper_sn_thm} and \ref{lower_s_thm}.

\begin{theorem}\label{upper_sn_thm}{\bf{(Upper bound for $W_n$ in the $Z^{s_n}$ case)}} Let $X_e\overset{d}{=}Z^{s_n}$ with $s_n\log n\to \infty$ and $s_n^2\log n\to 0$ for $n\to \infty$. Then
\[
W_n \le e \eta(s_n \log n) u_n \qquad \text{with high probability},
\]
where $\eta(x)=\min\{\lceil x\rceil, e^{\frac{x}{\lfloor x \rfloor }-1}\lfloor x \rfloor \}$ for $x \ge 1$. (Note that $\eta(x) \in [\lfloor x \rfloor, \lceil x \rceil]$.)
\end{theorem}

\begin{remark}\label{lambda_one_remark}
It suffices to prove statements about weight and hopcount in the case $\lambda=1$. In fact, the random variable $\tilde{Z}=\lambda Z$ has distribution function $\tilde{G}(x)=G(\lambda^{-1}x)$ and, therefore, satisfies the assumptions with $\tilde{G}'(0+)=1$. Moreover, the optimal path in the graph with edge weights $\tilde{X}_e=\lambda^{s_n} X_e$ agrees with the optimal path in the original network. Hence, its weight in the rescaled network is given by $\lambda^{s_n} W_n$, where $W_n$ is the optimal weight in the original graph. Since $\tilde{u}_n:=\tilde{G}^{-1}(1/n)^{s_n} =\lambda^{s_n} G^{-1}(1/n)^{s_n}=\lambda^{s_n}u_n$, our statements can indeed be derived from the $\lambda=1$ case.
\end{remark}

For the proof, we write $k_n^- = \lfloor s_n \log n\rfloor$ and $k_n^+ =\lceil s_n \log n \rceil$. By $k_n$ we denote an element of $\{k_n^-, k_n^+\}$ and
\[
\sigma_n=\begin{cases}
1& \text{if } k_n=k_n^+\\
e^{\frac{s_n \log n}{k_n}-1} &  \text{if }k_n=k_n^-.\end{cases}
\]
Hence, $\sigma_nk_n =\eta(s_n\log n)$ and $b_n= e \sigma_n k_n u_n$ is the claimed upper bound. Recall from (\ref{unscaling}) that $u_n=(\lambda n)^{-s_n} [1+O(1/n)]^{s_n}$.

\begin{proof}[Proof of \refthm{upper_sn_thm}]
We assume that $\lambda=1$. To apply \refprop{gen_upper_prop}, we first show the divergence of the expected value of $N_{k_n}(b_n)$. By assumption, $(b_n)$ is a null sequence and $k_n=o(p_n)$. Therefore (\ref{mean_estimate}), \reflemma{asymp_sn_lem}, $a_{k_n} \to \infty$ and (\ref{unscaling}) imply that
\begin{equation}\label{mean_upper_eq}
\begin{split}
\E[N_{k_n}(b_n)]&=
(1+o(1)) n^{k_n-1} \Big(\frac{b_n}{k_n}\Big)^{k_np_n} p_n^{(k_n-1)/2}a_{k_n} \ge n^{k_n-1}\big(e\sigma_n u_n\big)^{k_n p_n} p_n^{(k_n-1)/2}\\
&\ge \exp\big(-\log n+k_np_n\log(e\sigma_n [1+O(1/n)]^{s_n})+\tfrac{k_n}{4} \log p_n\big).
\end{split}
\end{equation}
By definition, $-\log n+k_np_n\log(e\sigma_n) \ge 0$. Hence, the right-hand side of (\ref{mean_upper_eq}) tends to infinity.

For the variance estimate, we apply \reflemma{asymp_sn_lem} and use that $a_k$ is increasing in $k$ to obtain for $l \in [k_n-2]$ and $n$ large
\begin{align*}
\frac{F^{*(k_n-l)}_n(b_n)}{F^{*k_n}_n(b_n)} &= (1+o(1)) \frac{\big(\frac{b_n}{k_n-l}\big)^{(k_n-l)p_n} p_n^{(k_n-l-1)/2} a_{k_n-l}}{\big(\frac{b_n}{k_n}\big)^{k_np_n} p_n^{(k_n-1)/2} a_{k_n}}\\
& \le b_n^{-lp_n} \Big(\frac{k_n^{k_n}}{(k_n-l)^{k_n-l}}\Big)^{p_n} p_n^{-l/2}
\le p_n^{-1/2} b_n^{-lp_n} \Big(\frac{k_n^{k_n}}{(k_n-l)^{k_n-l}}\Big)^{p_n}\\
&\le \sqrt{s_n} \exp(-lp_n \log b_n +k_np_n \log k_n-(k_n-l)p_n\log(k_n-l))\\
&= \sqrt{s_n} \exp(lp_n\log \tfrac{k_n}{b_n}+(k_n-l)p_n \log \tfrac{k_n}{k_n-l}).
\end{align*}
The definitions of $b_n$, $p_n$ and $\sigma_n$ combined with (\ref{unscaling}) imply that
\begin{align*}
&l [p_n \log \tfrac{k_n}{b_n}-\log n]+(k_n-l)p_n\log \tfrac{k_n}{k_n-l}=-l [p_n\log(e \sigma_n)+O(1/n)]+(k_n-l)p_n \log \tfrac{k_n}{k_n-l}\\
&\qquad\le lp_n\big[-1+(1/q_{l,n}-1)\log (1-q_{l,n})^{-1}+O(s_n/n)\big]= l p_n\big[\psi(q_{l,n})-1+O(s_n/n)\big],
\end{align*}
where we write $q_{l,n}:=l/k_n$ and $\psi(q)= (1/q-1)\log (1-q)^{-1}$. The following lemma summarizes properties of $\psi$:

\begin{lemma} \label{psi_lemma}
The mapping $\psi:(0,1) \to \mathbb{R}$ given by $\psi(q)= (1/q-1)\log (1-q)^{-1}$ can be continuously extended to $[0,1]$ by $\psi(0):=1$, $\psi(1):=0$. Moreover, $\psi$ is strictly decreasing in $[0,1]$ and $\psi(q)=1-\frac{1}{2} q+O(q^2)$ for $q\downarrow 0$.
\end{lemma}

\begin{proof}[Proof of \reflemma{psi_lemma}]
We Taylor expand the logarithm to deduce
\begin{align*}
\psi(q)&=-\tfrac{1}{q} (1-q) \log(1-q)=-\tfrac{1}{q} (1-q) [-q-\tfrac{1}{2}q^2 +O(q^3)]\\
&=(1-q)[1+\tfrac{1}{2}q+O(q^2)]=1-\tfrac{1}{2}q+O(q^2) \qquad \text{for } q\downarrow 0.
\end{align*}
Moreover, $\lim_{q\to 1} \psi(q)=-\lim_{q\to 1}(1-q)\log(1-q)=0$. Finally, $1-x<e^{-x}$ for all $x\not=0$ implies that
\[
\psi'(q)=\frac{1}{q^2}\big(q+\log(1-q)\big)<0 \qquad\quad \forall \,q\in (0,1).\qedhere
\]
\end{proof}

Since $\psi$ is decreasing and $\psi(q_{l,n})\le \psi(1/k_n)=1-1/(2k_n)+O(k_n^{-2})$, we can use $s_n/n=o(1/k_n)$ and $k_n/p_n=s_n k_n=o(1)$ to estimate
\begin{align*}
\sum_{l=1}^{k_n-2}(l+1)n^{-l}\frac{F^{*(k_n-l)}_n(b_n)}{F^{*k_n}_n(b_n)}&\le \sqrt{s_n} \sum_{l=1}^{\infty} (l+1) \exp\Big(-\frac{lp_n}{2k_n}(1+o(1))\Big)\\
&= 2(1+o(1)) \sqrt{s_n}\exp\big(-\tfrac{p_n}{2k_n}(1+o(1))\big)=o(1).
\end{align*}
This proves the second condition in \refprop{gen_upper_prop} and \refthm{upper_sn_thm} follows.
\end{proof}

Our methods allow to prove (\ref{Wn_in_prob_eq}) under the weaker assumption that $G\in C^1([0,\hat{x}])$ with $G(x)=\lambda x+O(x^2)$ for $x \downarrow 0$ when $s_n^2 \log n \log \log n\to 0$. To show the upper bound, one can use similar ideas to the proof of \refthm{upper_sn_thm} but employ the estimates from \reflemma{rough_sn_lem} together with the techniques used in the proof of \refthm{weight_a_thm} below. The proof of the lower bound is valid under these weaker assumptions without any changes as we will see in the next theorem. We keep the $C^2$-assumption to get better upper bounds on the weight and upper and lower bounds on the hopcount in \refsect{hop_sec}. These bounds allow us to prove a concentration of the hopcount under additional assumptions on $(s_n)$ in \refcoro{concentration_coro}.

We turn to the lower bound. Here the assumption $s_n^2 \log n \to 0$ is not needed.

\begin{theorem}\label{lower_s_thm}{\bf{(Lower bound for $W_n$ in the $Z^{s_n}$ case)}}
Let $X_e\overset{d}{=}Z^{s_n}$ with $s_n\to 0$ and $s_n \log n \to \infty$. Then
\[
W_n \ge (1-\epsilon_n) e \lfloor s_n \log n \rfloor u_n \qquad \text{with high probability}
\]
for every sequence $(\epsilon_n)_{n\in \mathbb{N}}\in (0,\infty)^{\N}$ with $\lim_{n\to \infty}\frac{ s_n \log 1/s_n}{\epsilon_n}=0$.
\end{theorem}

Note that in case $\lim_{n\to \infty}s_n \log n=\infty$, $\lim_{n\to \infty}s_n^2 \log n=0$ the condition on $\eps_n$ is equivalent to $s_n \log \log n=o(\eps_n)$, because for $n$ large enough
\begin{align*}
\log s_n&=\log( \sqrt{s_n^2 \log n} (\log n)^{-\frac{1}{2}}) \le -\tfrac{1}{2} \log \log n,\\
\log s_n&=\log(s_n \log n(\log n)^{-1})\ge -\log \log n.
\end{align*}
Thus,
\begin{equation}\label{logpn_control}
\frac{1}{2}\le\frac{\log p_n}{\log\log n}\le 2  \qquad \text{for $n$ sufficiently large.}
\end{equation}

For the proof of \refthm{lower_s_thm}, we write $k_n^{-}=\lfloor s_n \log n\rfloor$ and $d_n:= (1-\epsilon_n) e k_n^- u_n$.

\begin{proof}[Proof of \refthm{lower_s_thm}]
Without loss of generality, assume that $\lambda=1$ and that $(\eps_n)$ is a null sequence. To check the criterion from \refprop{lower_gen_prop}, let $c>0$ be the constant from \reflemma{rough_sn_lem}. For all sufficiently large $n$,
\[
\sum_{k=1}^{n-1} n^{k-1} F_n^{*k}(d_n) \le \sum_{k=1}^{n-1} (cp_n)^k \exp\big((k-1)\log n +kp_n \log \tfrac{d_n}{k}\big).
\]
The term in the exponent can be rewritten as
\begin{align}\label{exponent_lower_bound}
(k-1)&\log n+kp_n \log(\tfrac{k_n^-}{k}(1-\eps_n)e n^{-s_n}[1+O(1/n)]^{s_n})
\notag\\
&=-p_ns_n \log n+kp_n\big[\log \tfrac{k_n^-}{k}+\log(1-\eps_n)+1+O(s_n/n)\big]
\notag\\
&= kp_n\big[-\tfrac{s_n \log n}{k} +\log \tfrac{s_n \log n}{k}+1+\log(1-\eps_n)+\log \tfrac{k_n^-}{s_n \log n}+O(s_n/n)\big]
\notag\\
&=kp_n\big[\varphi(\tfrac{s_n\log n}{k})+1+\log(1-\eps_n)+\log \tfrac{k_n^-}{s_n \log n}+O(s_n/n)\big],
\end{align}
where $\varphi(x):= -x+\log x$. Since $\varphi$ has its global maximum at $x=1$ and as $s_n/n=o((\log p_n)/p_n)=o(\eps_n)$, we conclude
\begin{align*}
\sum_{k=1}^{n-1} n^{k-1} F_n^{*k}(d_n) &\le  \sum_{k=1}^{n-1} (cp_n)^k \exp\big(k p_n [\log(1-\epsilon_n)+O(s_n/n)]\big)\\
&\le \sum_{k=1}^{\infty} \exp\big(-kp_n\eps_n\big[1+o(1) - \tfrac{\log(cp_n)}{\eps_n p_n}\big]\big)\\
&=(1+o(1))\exp\big(-p_n\epsilon_n[1+o(1)]\big)=o(1).\qedhere
\end{align*}
\end{proof}

\subsection{The hopcount}\label{hop_sec}
Knowing an upper bound for the weight of the smallest weight path, we can give estimates for the hopcount.

\begin{theorem}\label{hop_sn_thm}{\bf{(Hopcount in the $Z^{s_n}$ case)}}
Let $X_e\overset{d}{=}Z^{s_n}$ with $s_n \log n \to \infty$ and $s_n^2 \log n \to 0$. Then
\[
\frac{1}{1+\beta_n} s_n \log n < H_n < \frac{1}{1-\beta_n} s_n \log n \qquad \text{with high probability}
\]
for all sequences $(\beta_n)_{n\in \mathbb{N}} \in (0,\infty)^{\mathbb{N}}$ satisfying $\limsup_{n\to \infty} \frac{s_n \log 1/s_n}{ \beta_n^2} \leq 1/2-\delta$ for some $\delta>0$ and $\lim_{n\to \infty} (\beta_n s_n \log n)^{-1} (s_n \log n -k_n^-) = 0$.
\end{theorem}

Notice that the second condition on $\beta_n$ is no restriction when $s_n \log n$ is integer-valued for all sufficiently large $n$. The convergence $\frac{H_n}{s_n \log n} \overset{\P}{\rightarrow} 1$ stated in \refthm{weight_thm} follows easily by choosing $\beta_n=\eps$ for arbitrary $\eps >0$.

\begin{proof}[Proof of \refthm{hop_sn_thm}] Without loss of generality, assume that $\lambda=1$ and that $(\beta_n)$ is a null sequence. Denote $k_n^- = \lfloor s_n\log n\rfloor$, $\sigma_n=e^{\frac{s_n \log n}{\lfloor s_n\log n\rfloor}-1}$ and $b_n = e \sigma_n k_n^- u_n$. By \refthm{upper_sn_thm}, $\P(W_n > b_n)\to 0$ for $n\to \infty$. Denote $h_n:=\lfloor \frac{1}{1+\beta_n} s_n \log n\rfloor$. To prove that $H_n > h_n$ with high probability, it is sufficient to show that
\[
\sum_{k=1}^{h_n} \mathbb{E}[N_k(b_n)] \to 0 \qquad \quad \text{for }n \to \infty.
\]
Indeed, this convergence implies
\begin{equation}\label{hopcount_lower_gen}
\begin{split}
\P(H_n\le h_n) &= \P(H_n \le h_n, W_n \le b_n) +o(1)\le \P\big(\cup_{k=1}^{h_n} \{N_k(b_n)\not=0\}\big) +o(1)\\
&\le \sum_{k=1}^{h_n}\P(N_k(b_n)\ge 1) +o(1)\le \sum_{k=1}^{h_n} \E[N_k(b_n)] +o(1)=o(1).
\end{split}
\end{equation}
Since $b_n =o(1)$, we can use \reflemma{rough_sn_lem} to obtain
\[
\sum_{k=1}^{h_n} \E[N_k(b_n)] \le \sum_{k=1}^{h_n} n^{k-1}F^{*k} (b_n) \le \sum_{k=1}^{h_n} n^{k-1} \Big(\frac{b_n}{k}\Big)^{kp_n} (cp_n)^k.
\]
Using (\ref{exponent_lower_bound}) with $\sigma_n$ in place of $1-\eps_n$, we see that
\[
(k-1)\log n+k p_n \log(\tfrac{k_n^-}{k}\sigma_ne u_n) =kp_n\big[\varphi(\tfrac{s_n \log n}{k})+1+\log \sigma_n+\log \tfrac{k_n^-}{s_n \log n}+O(s_n/n)\big].
\]
The function $\varphi(x)= -x +\log x$ is strictly concave with unique maximum at $x=1$. Since $k\in [h_n]$ and $h_n\le \frac{1}{1+\beta_n} s_n \log n$, we have $\varphi(\frac{s_n \log n}{k})+1 \le \varphi(1+\beta_n)+1=-\frac{1}{2}\beta_n^2(1+o(1))$. Moreover, $O(s_n/n)=o((\log p_n)/p_n)=o(\beta_n^2)$ and by assumption
\[
\log \sigma_n+\log \tfrac{k_n^-}{s_n \log n}=
\tfrac{s_n \log n}{k_n^-}-1- \big(\tfrac{s_n \log n-k_n^-}{k_n^{-}}+O((\tfrac{s_n \log n-k_n^-}{k_n^{-}})^2)\big)=
O((\tfrac{s_n \log n-k_n^-}{k_n^{-}})^2)=o(\beta_n^2).
\]
Inserting this estimate, we derive
\[
\sum_{k=1}^{h_n} \E[N_k(b_n)]\le
\sum_{k=1}^{h_n} \exp\big(-kp_n\beta_n^2 [\tfrac{1}{2}+o(1)-\tfrac{\log(cp_n)}{p_n \beta_n^2}]\big)\le
\sum_{k=1}^{\infty}\exp\big(-k p_n\beta_n^2 [\delta+ o(1)]\big)=o(1).
\]
We now turn to the upper bound. Let $h_n:= \lceil \frac{1}{1-\beta_n}s_n \log n\rceil$. Similar to the lower bound, it suffices to show that
\[
\sum_{k=h_n}^{n-1} \mathbb{E}[N_k(b_n)] \to 0 \qquad \quad \text{for }n \to \infty.
\]
With the same estimates as above we get
\begin{align*}
\sum_{k=h_n}^{n-1} \mathbb{E}[N_k(b_n)]&\le \sum_{k=h_n}^{n-1} \exp\Big(kp_n\big[\varphi(\tfrac{s_n \log n}{k})+1+o(\beta_n^2)+\tfrac{\log(cp_n)}{p_n}\big]\Big).
\end{align*}
Since $k \ge h_n \ge \frac{1}{1-\beta_n} s_n \log n$, we can again use the strict concavity of $\varphi$ but now applied for values smaller than $1$. Hence, $\varphi(\frac{s_n \log n}{k})+1\le \varphi(1-\beta_n)+1=-\frac{1}{2} \beta_n^2 (1+o(1))$. Therefore, we have the exact same exponent as in the lower bound and the statement is established.
\end{proof}

Under stronger assumptions on the sequence $(s_n)$, we obtain the following concentration result for the hopcount:

\begin{corollary}{\bf{(Concentration of the hopcount)}}\label{concentration_coro}
Let $X_e\overset{d}{=}Z^{s_n}$ with $s_n^{3} (\log n)^2 \log \log n \to 0$ and $s_n \log n \to \infty$. If $k_n = s_n \log n$ is integer-valued for large $n$, then
\[
\P(H_n =k_n ) \to 1 \qquad \text{for }n \to \infty.
\]
\end{corollary}

\begin{proof}
Choose $\beta_n = 2\sqrt{\frac{\log p_n}{p_n}}$ in \refthm{hop_sn_thm}. We have
\[
\frac{k_n}{1+\beta_n} < H_n < \frac{k_n}{1-\beta_n} \quad \Leftrightarrow \quad -\frac{\beta_n k_n}{1+\beta_n} < H_n-k_n < \frac{\beta_n k_n}{1-\beta_n}.
\]
Using the definition of $\beta_n$ and $k_n$, (\ref{logpn_control}) and the assumption, we obtain
\[
k_n \beta_n = s_n \log n \,2\sqrt{\tfrac{\log p_n}{p_n}}= 2 s_n^{3/2} \log n \,(\log \log n)^{1/2} \sqrt{\tfrac{\log p_n}{\log \log n}}= o(1).
\]
Hence, $|H_n -k_n|\overset{\P}{\rightarrow} 0$. Since $H_n$ is integer-valued, this implies the claim.
\end{proof}

\subsection{The very small \texorpdfstring{$s_n$}{sn} case}\label{very_small_sn}

In this section we determine the behaviour of the weight of the shortest path and the hopcount in the case that $(s_n)$ tends to zero like $\gamma/ \log n$. We begin with rough estimates to prove \refthm{very_small_sn_thm}. These estimates are then improved, under slightly stronger conditions, to establish \refthm{very_small_2nd_order}.

First, assume that $Z$ is a positive random variable with $G(x)=\P(Z \le x)=\lambda x (1+o(1))$ for some $\lambda >0$ and $x\downarrow 0$. In the sequel we write $p_n=1/s_n$ and for a vertex $v \in [n]$ and an edge $e \in \text{E}_n$, $e \sim v$ means that $v$ is a vertex in $e$.

\begin{proof}[Proof of \refthm{very_small_sn_thm}] For the upper bound on $W_n$, note that by definition $W_n \le X_{\{1,2\}}\overset{d}{=}Z^{s_n}$ and $Z^{s_n} \overset{\P}{\longrightarrow} 1$. For the lower bound, we use that $F_n(x)=G(x^{p_n}) =o(1)$ for $x<1$ to obtain for those $x$
\begin{align}\label{bound_minE}
\P\Bigl( \min_{e \in \text{E}_n, e \sim 1} X_e \ge x \Big)
&=\big(1-F_n(x)\big)^{n-1} =\exp\big(-(n-1) F_n(x) (1+o(1))\big)
\notag\\&=
\exp\big(-(n-1)^{s_n p_n} \lambda x^{p_n} (1+o(1))\big)
\notag\\&=
\exp\big(- \lambda \big[e^{\gamma} x+o(1)\big]^{p_n} (1+o(1))\big),
\end{align}
where we used that $(n-1)^{s_n}=e^{\gamma} (1+o(1))$. The right-hand side converges to $1$ for all $x < e^{-\gamma}$. Since $W_n \ge \min_{e \in \text{E}_n, e \sim 1}X_e$, we conclude
\[
\P\big(W_n \ge (1-\eps)e^{-\gamma}\big)\ge \P\big( \min_{e \in \text{E}_n, e \sim 1} X_e \ge (1-\eps)e^{-\gamma}\big) \to 1.
\]
For the hopcount, we note that
\[
H_n \le \frac{W_n}{\min_{e \in \text{E}_n} X_e} \le \frac{X_{\{1,2\}}}{\min_{e \in \text{E}_n} X_e}.
\]
Choose $\delta>0$ so that $1+\epsilon= (1+\delta)^2$. 
A computation similar to (\ref{bound_minE}), with $n-1$ replaced by $\binom{n}{2}$, shows that $\min_{e \in \text{E}_n}X_e \ge x$ with high probability for all $x < e^{-2\gamma}$. Thus,
\begin{align*}
\P(H_n \le (1+\epsilon) e^{2\gamma}) &\ge \P( X_{\{1,2\}}\big(\min_{e \in \text{E}_n} X_e\big)^{-1}\le (1+\delta)^2 e^{2\gamma})\\
&\ge \P(\{X_{\{1,2\}} \le 1+\delta \}\cap \{\big(\min_{e \in \text{E}_n} X_e\big)^{-1} \le (1+\delta)e^{2\gamma}\})= 1-o(1).
\end{align*}
Since $\eps>0$ was arbitrary and $H_n$ is integer-valued, the claim is established.
\end{proof}

The bounds in \refthm{very_small_sn_thm} are not optimal. To derive the correct values for $W_n$ and $H_n$, we assume that there exists $\hat{x} \in (0,\infty)$ such that $G \in C^1([0,\hat{x}])$ and $G'(0+)>0$. Recall from (\ref{eq:defgg}) that we define for $\gamma \in [0,\infty)$ and $x >0$
\[
\gga(x)=x e^{-\gamma (1-1/x)}= e^{-\gamma} x e^{\gamma/x}.
\]
Let $W_n(k):=\min_{\p \in \S_{1,2}(k)} w(\p)$ denotes the weight of the shortest path between vertex $1$ and $2$ which uses exactly $k$ edges. The following proposition provides lower bounds on $W_n(k)$.

\begin{proposition}\label{prop:Wnk}
Let $X_e\overset{d}{=}Z^{s_n}$ with $s_n\log n  \to \gamma \in [0,\infty)$. \\
(a) For all $\eps >0$,
\[
W_n(k) \ge (1-\eps) \min\{\gga(k),1\} \qquad \forall\, k \in \N \text{ with high probability}.
\]
(b) Let $k \ge 2$ with $\gga(k)>1$ and $\eps<\min\{e^{-\gamma (k-1)}/4,\gga(k)-1\}$. Then $W_n(k) \ge 1+\eps$ whp.
\end{proposition}

\begin{proof}
Before we start the main part of the proof, notice that for all $\eps \in (0,1)$,
\begin{equation}\label{ggProp}
n^{k-1} \Big(\frac{(1-\eps) \gga(k)}{k}\Big)^{k p_n} =\exp\Big(p_n k \big[\log(1-\eps) + (1-1/k) (s_n \log n -\gamma)\big]\Big) \le e^{-k \eps p_n (1+o(1))},
\end{equation}
where the error is uniform in $k \in \N$.\\
(a) Denote $x_{k}=(1-\eps) \min\{\gga(k),1\}$. Since $x_k<1$, $x_k^{p_n}<\hat{x}$ for sufficiently large $n$. By \reflemma{rough_sn_lem},
\[
F_n^{*k}(x_k) \le (cp_n)^k \Big(\frac{x_k}{k}\Big)^{kp_n},
\]
where $c$ is the constant from the Lemma. Hence, (\ref{ggProp}) implies that
\begin{align*}
\E[N_k(x_k)]&\le (cp_n)^k n^{k-1} \Big(\frac{x_k}{k}\Big)^{kp_n} \le \exp\big( - k \eps p_n (1+o(1)) \big),
\end{align*}
where the error is uniform in $k \in \N$. In particular, there exists $n_0 \in \N$ such that $\E[N_k(x_k)] \le e^{-2 \eps p_n k}$ for all $k \in \N$, $n \ge n_0$. Markov's inequality yields for all $n \ge n_0$
\begin{align*}
\P\Big(\bigcup_{k \in \N} \{W_n(k) < x_k\}\Big)& \le \sum_{k =1}^{\infty} \P(W_n(k) <x_k) \le \sum_{k=1}^{\infty} \P(N_k(x_k) \ge 1)\\
&\le\sum_{k=1}^{\infty} \E[N_k(x_k)] \le \sum_{k=1}^{\infty} e^{-2 \eps p_n k} = \frac{e^{-2 \eps p_n} }{1-e^{-2 \eps p_n}}=o(1).
\end{align*}
(b) Let $0< \eps< \min\{e^{-\gamma (k-1)}/4,\gga(k)-1\}$ and $\delta=e^{-\gamma (k-1)}/2$. We first show that every path between vertex $1$ and $2$ which uses exactly $k$ edges contains with high probability only edges of weight larger than $\delta$. Indeed,
\begin{equation}\label{XeLarge}
\begin{split}
\P\Big( \bigcup_{\p \in \S_{1,2}(k)} \bigcup_{e \in \p} \{X_e \le \delta\}\Big) &\le n^{k-1} k G(\delta^{p_n})= n^{k-1} k \lambda \delta^{p_n} (1+o(1))\\
&=(1+o(1)) k \lambda \exp\big((k-1) \log n- p_n \gamma (k-1) -p_n \log 2 \big)=o(1).
\end{split}
\end{equation}
In the next step we estimate the probability that the weight of a chosen path $\p \in \S_{1,2}(k)$ is at most $1+\eps$. By assumption $d=\max\{G'(x): x \le 1\land \hat{x}\} \in (0,\infty)$. Moreover, when $x_1^{s_n}+ \ldots +x_k^{s_n} \le 1+\eps$ and $x_i^{s_n} > \delta$ for all $i \in [k]$, then $k\ge 2$ implies that $x_i^{s_n} \le 1+\eps -\delta \le 1-\eps$ for all $i \in [k]$. In particular, $x_i \le (1-\eps)^{p_n} \le 1\land \hat{x}$ for $n$ sufficiently large. Thus,
\begin{align*}
&\P\Big(\sum_{e \in \p} X_e \le 1+\eps, X_e > \delta \, \forall e \in \p\Big) =\int_{ \sum_{i=1}^k x_i^{s_n} \le 1+\eps,\, x_i^{s_n}>\delta \, \forall i \in [k]} \prod_{i=1}^k G'(x_i) \; d(x_1,\ldots x_k)\\
&\qquad \qquad \qquad \le d^k \int_{ \sum_{i=1}^k x_i^{s_n} \le 1+\eps,\, x_i \ge 0 \, \forall i \in [k]} \; d(x_1,\ldots x_k) = d^k (1+\eps)^{k p_n} \frac{\Gamma(1+p_n)^k}{\Gamma(1+kp_n)},
\end{align*}
where we used 4.634 in \cite{GraRyz94} to compute the integral. Since $1+ \eps <\gga(k)$, (\ref{stirlingForFk}) and (\ref{ggProp}) yield
\begin{equation}\label{1pathweight}
\P\Big( \bigcup_{\p \in \S_{1,2}(k)} \Big\{ \sum_{e \in \p} X_e \le 1+\eps, X_e > \delta \, \forall e \in \p \Big\}\Big) \le n^{k-1} d^k (1+\eps)^{p_n k} \frac{\Gamma(1+p_n)^k}{\Gamma(1+kp_n)} = o(1).
\end{equation}
The sum of (\ref{XeLarge}) and (\ref{1pathweight}) provides an upper bound on the probability that $W_n(k)$ is at most $1+\eps$. Thus, the proof is complete. 
\end{proof}

\refprop{prop:Wnk} immediately gives improved bounds on the weight of the optimal path and the hopcount. We denote $\theta(\gamma)=\sup\{k \in \N: \gga(k) \le 1\}$.

\begin{theorem}\label{thm:VsmallC1}
Let $X_e\overset{d}{=}Z^{s_n}$ with $s_n\log n  \to \gamma \in [0,\infty)$. Then, for all $\eps >0$,
\[
(1-\eps) \min_{k \in \N} \gga(k) \le W_n \le 1+ \eps \qquad \text{and} \qquad 1 \le H_n \le \theta(\gamma) \qquad \text{with high probability.}
\]
\end{theorem}

\begin{proof}
By definition $W_n=\min_{k \in \N}W_n(k)$. So the claim for $W_n$ follows from \refprop{prop:Wnk}(a) and the fact that $\min_{k \in \N}\gga(k) \le \gga(1)=1$. Since by \refthm{very_small_sn_thm}, $W_n \le 1+\eps$ for all $\eps>0$ and $H_n \le e^{2\gamma}$ with high probability, \refprop{prop:Wnk}(b) applied to $k \in \{\theta(\gamma)+1,\ldots, e^{2\gamma}\}$ yields the statement for the hopcount.
\end{proof}

\begin{proof}[Proof of \refthm{very_small_2nd_order} (a) and (b)]
As mentioned below (\ref{eq:defgg}), for $\gamma \in [0,2\log 2]$, the minimum of $\gga$ on $\N$ equals one, for $\gamma < 2\log 2$ we have $\theta(\gamma) =1$ and for $\gamma =2\log 2$ we have $\theta(\gamma)=2$. Hence, \refthm{thm:VsmallC1} yields the statements for the hopcount and that $W_n\overset{\P}{\longrightarrow} 1$. In particular, for $\gamma \in [0,2 \log 2)$ we have $W_n=X_{\{1,2\}}$ with high probability. Since $X_{\{1,2\}}\overset{d}{=}Z^{s_n}$ and
\begin{equation}\label{Hn1WeakConv}
s_n^{-1}(Z^{s_n}-1)= s_n^{-1} (e^{s_n \log Z}-1) \rightarrow \log Z \qquad \text{a.s.},
\end{equation}
the claim is established. 
\end{proof}

\refthm{very_small_2nd_order} reveals that for $\gamma <2 \log 2$, the distribution of $Z$ can be retrieved from the weak limit of a linear transformation of $W_n$. Hence, the entire distribution function plays a role. This is in contrast to the case $\gamma>2\log 2$ which we will now investigate. Here, only the behaviour of the distribution close to zero is relevant and a linear transformation of the $W_n$ converges to a Gumbel distribution. The Gumbel distribution usually arises in minimizations over independent random variables, and also here we will see that we minimize over almost 
independent paths of a fixed length. When the hopcount equals one (i.e.\ if $\gamma < 2 \log 2$), there is only one possible path and the optimization is trivial. This explains the different behaviour for the two regimes. However, we remark that when $Z$ is exponentially distributed, then $-\log Z$ still follows a Gumbel distribution.

In order to establish \refthm{very_small_2nd_order} for $\gamma >2 \log 2$, we assume from now on that $G \in C^2([0,\hat{x}])$, in order to have the asymptotics for the distribution function derived in (\ref{detailed_df_sn}) at our disposal. 

\begin{proof}[Proof of \refthm{very_small_2nd_order} (c) and (d)]
Since the proof follows the argument in \cite{BhaHofHoo12pre} very closely, we will only give a sketch. According to Remark \ref{lambda_one_remark} it suffices to prove the statement for $\lambda=1$. We denote
\[
W_n^{(\text{ind})}(k)=\min_{i \in [n^{k-1}]} Y_{i,k}
\]
with $Y_{1,k} \ldots Y_{n^{k-1},k}$ independent and identically distributed with distribution function $F_n^{*k}$. Then
\begin{equation}\label{WnkIndAsymp}
\P\Big( \frac{k}{\gga(k)} p_n\big[W_n^{(\text{ind})}(k)-\gga(k)\big] + (k-1) p_n \big(s_n \log n-\gamma +\frac{\log p_n}{2p_n} \big)>t \Big) \rightarrow e^{-a_k e^t}
\end{equation}
for all $k \ge 2$ with $\gga(k)<1$. Indeed, let
\begin{equation}\label{def:zknt}
z_{k,n}(t):=\gga(k) \Big(1+\frac{1}{k} \big[s_n t -(k-1) (s_n \log n -\gamma +\frac{\log p_n}{2p_n})\big]\Big)=\gga(k) (1+y_{k,n}(t)),
\end{equation}
where $y_{k,n}(t)=o(1)$. Then (\ref{WnkIndAsymp}) will follow once we show that
\begin{equation}\label{WnkInd:ts}
\P(W_n^{(\text{ind})}(k) >z_{k,n}(t) )=(1-F_n^{*k}(z_{k,n}(t)))^{n^{k-1}} \to e^{-a_k e^t}.
\end{equation}
Since $z_{k,n}(t) \to \gga(k)<1$,
\[
F_n^{*k}(z_{k,n}(t)) \sim a_k \Big(\frac{z_{k,n}(t)}{k}\Big)^{k p_n} p_n^{(k-1)/2}.
\]
In particular, (\ref{WnkInd:ts}) is equivalent to
\begin{equation}\label{meanPoisConv}
n^{k-1}F_n^{*k}(z_{k,n}(t)) \to a_k e^t.
\end{equation}
Taking logarithms and using the definition of $z_{k,n}(t)$ and $\log(1+y_{k,n}(t))=y_{k,n}(t) + O(y_{k,n}(t)^2)$, this is equivalent to
\[
p_n[s_n t +O(y_{k,n}(t)^2)]\to t,
\]
what is guaranteed by the assumption $\sqrt{\log n}(\gamma - s_n \log n)=0$.  We have therefore proved (\ref{WnkIndAsymp}).

In the next step we approximate the number of paths between vertex $1$ and $2$ of length $k$ and weight at most $z_{k,n}(t)$ by a Poisson random variable with the same mean. That is, we show that for
\[
\lambda_{k,n}(t)=\E[N_{k}(z_{k,n}(t))] = \prod_{j=2}^{k} (n-j) F_n^{*k}(z_{k,n}(t)) \sim a_k e^t,
\]
and $k \in \{\lfloor \gamma \rfloor, \lceil \gamma \rceil\}$ with $\gga(k)<1$, we have
\begin{equation}\label{PoisApproxTV1}
d_{\text{TV}}\big(N_k(z_{k,n}(t)), \text{Pois}(\lambda_{k,n}(t))\big) \rightarrow 0,
\end{equation}
where $d_{\text{TV}}$ denotes the total variation. In Proposition 4.4 of \cite{BhaHofHoo12pre} it is shown that given (\ref{meanPoisConv}) it suffices to check that $n^{2k-j-2} p_{j,k}^{(n)}(t) \to 0$ for all $1 \le j \le k-2$, where
\[
p_{k,j}^{(n)}(t) =\P\Big(\sum_{i=1}^k Z_i^{s_n} \le z_{k,n}(t), \sum_{i=1}^j Z_i^{s_n} + \sum_{i=j+1}^k \tilde{Z}_i^{s_n} \le z_{k,n}(t)\Big)
\]
and $Z_1,\ldots,Z_k,\tilde{Z}_1,\ldots \tilde{Z}_k$ are independent copies of $Z$. Let us first fix $z \in (0,1)$. Using (\ref{detailed_df_sn}) and Stirling's formula we find that
\begin{align*}
\P\Big(\sum_{i=1}^k Z_i^{s_n} \le z, \sum_{i=1}^j Z_i^{s_n} + \sum_{i=j+1}^k &\tilde{Z}_i^{s_n} \le z\Big) = \int_0^z F_n^{*(k-j)}(z-y)^2 \, dF_n^{*j}(y)\\
& \quad \sim \frac{\sqrt{2}\sqrt{2\pi}^{2k-j-2}}{\sqrt{k-j}\sqrt{2k-j}} \sqrt{p_n}^{2k-j-2} \Big(\frac{2^{2(k-j)}}{(2k-j)^{2k-j}}\Big)^{p_n}z^{(2k-j)p_n}.
\end{align*}
Thus, it remains to check that for all $j\in [k-2]$
\begin{align*}
&n^{2k-j-2} z_{k,n}(t)^{(2k-j)p_n} \Big(\frac{2^{2(k-j)}}{(2k-j)^{2k-j}}\Big)^{p_n} \sqrt{p_n}^{2k-j-2}\\
&\;=
\exp\Big( p_n \big[(2k-j-2)\gamma + (2k-j)\log \gga(k) + (k-j) \log 4- (2k-j) \log(2k-j) +o(1) \big]\Big)
\end{align*}
is a null sequence. Using the definition of $\gga$ and noting that $k \in \{\lfloor \gamma \rfloor, \lceil \gamma \rceil\}$, $k \ge 2$, one can check that $r(x)=(2k-x-2)\gamma + (2k-x)\log \gga(k) + (k-x) \log 4- (2k-x) \log(2k-x)$ is strictly decreasing on $[0,k-1]$. Since $r(0)=0$, we conclude that $r(1)<0$ and the proof of (\ref{PoisApproxTV1}) is complete.

As a consequence we obtain for $k \in \{\lfloor \gamma \rfloor, \lceil \gamma \rceil\}$ with $\gga(k)<1$,
\[
\P(W_n(k) >z_{k,n}(t))=\P(N_k(z_{k,n}(t))=0) \to e^{-a_k e^t}.
\]
Analysing $\gga$, we find that for $\gamma \ge 2$, $\gga(\lfloor \gamma \rfloor)$ and $\gga(\lceil \gamma \rceil)$ are both smaller than one. For $\gamma \in (2\log 2,2)$, $k(\gamma)=2$ and $\gga(2)<1$.\\
In particular, if $\gamma >2 \log 2$, $\gamma \not \in \Gamma$, $W_n(k(\gamma))$ converges to $\gga(k(\gamma))$ in probability. Since $\gga(k(\gamma)) < \gga(k)$ for all $k \not= k(\gamma)$, \refprop{prop:Wnk}(a) yields part (c) of \refthm{very_small_2nd_order}. If $\gamma = \gamma_k \in \Gamma$ for $k \ge 2$, then $W_n(k)$ and $W_n(k+1)$ both converge in probability to $g(k)=g(k+1)$ and \refprop{prop:Wnk}(a) shows that $W_n(l)$ is bounded from below by a strictly larger value for all $l \in \N\setminus \{k,k+1\}$. Thus, $W_n =\min\{W_n(k),W_n(k+1)\}$ with high probability and \refthm{very_small_2nd_order} (d) follows.
\end{proof}

\section{\texorpdfstring{A distribution with implicit $s_n$}{Distribution with implicit sn}}\label{alpha_sec}

In this section we consider the edge weights $X_e\overset{d}{=} e^{-(E/\rho)^{1/\alpha}}$ where $E$ is an exponential random variable with mean one and $\alpha>1$, $\rho>0$. The heuristics in \refsect{discussion} suggest that for $\alpha>2$ this edge weight distribution falls in the same universality class as $X_e=Z^{s_n}$ with $s_n \log n\to \infty$ and $s_n^2 \log n\to 0$. The implicit $s_n$ is given by $s_n=(\log n)^{-1+1/\alpha}/(\alpha \rho^{1/\alpha})$.

The proofs of this section follow the same lines as in \refsect{main_proofs_sec}, but with more technical difficulties 
because of the more complicated distribution function. We start with the required estimates. Since here the distribution does not depend on $n$, we write $F^{*k}$ instead of $F_n^{*k}$ for $k \in \N$.

\begin{lemma}\label{Fk_estimate_lem}
Let $\alpha >1$. There exists $\hat{x}=\hat{x}(\alpha,\rho)>0$ such that for all $k\in \mathbb{N}$ and $x\in (0, \hat{x})$
\[
\exp\big(-k\rho (\log \tfrac{k}{x})^{\alpha}\big) \le F^{*k}(x) \le (\rho \alpha e)^k (\log \tfrac{k}{x})^{k(\alpha-1)} \exp\big(-k\rho(\log \tfrac{k}{x})^{\alpha}\big).
\]
\end{lemma}

\begin{proof}
Notice that for $x\in (0,1)$
\[
F(x)=\P(e^{-(E/\rho)^{1/\alpha}} \le x)=\P(E\ge \rho (\log \tfrac{1}{x})^{\alpha})=\exp(-\rho (\log \tfrac{1}{x})^{\alpha}).
\]
Hence, the lower bound follows immediately from the lower bound in \refprop{gen_rough_prop}. For the upper bound, compute the Lebesgue density $f(x)$ of the distribution $F$:
\begin{equation}
f(x)=\frac{dF}{dx}(x)= \frac{\rho \alpha}{x} (\log \tfrac{1}{x})^{\alpha-1} \exp(-\rho( \log \tfrac{1}{x})^{\alpha}) \qquad \forall \, x \in (0,1).
\end{equation}
The constant factor $\rho \alpha$ can be ignored for the optimization. Hence, we choose $h(x):=\rho (\log \frac{1}{x})^{\alpha}-\log \frac{1}{x}-(\alpha-1)\log \log \frac{1}{x} $ and $\vartheta(x):=\log \frac{1}{x}=-\log x$ for $x \in (0,1)$ and $b(x):=\rho x^{\alpha}-x-(\alpha-1)\log x$ for $x \in (0,\infty)$. Then $h=b \circ \vartheta$ satisfies the assumptions of \reflemma{strict_convex_lem} with any $\bar{x}\in (0,1]$ for which we prove convexity in $(0,\bar{x})$. The assumption $\alpha >1$ implies that
\[
b'(x)=\rho \alpha x^{\alpha-1}-1-(\alpha-1)/x
\]
is strictly increasing, $\lim_{x \downarrow 0}b'(x)=-\infty$ and $\lim_{x\to \infty}b'(x)=\infty$. Hence, there exists a unique $\tilde{x}=\tilde{x}(\alpha, \rho)>0$ such that $b'(\tilde{x})=0$. Thus, $b$ is convex and strictly increasing on $[\tilde{x},\infty)$. As $\vartheta$ is strictly convex and $\vartheta(x)\ge \tilde{x}$ for $x \in (0,\vartheta^{-1}(\tilde{x})]$, the composition $b \circ \vartheta = h$ is strictly convex on this interval. Moreover, the unique stationary point of $h$ is $x^*=\vartheta^{-1}(\tilde{x})=e^{-\tilde{x}}$. As a result, \refprop{gen_rough_prop} and \reflemma{strict_convex_lem} with $\bar{x}=x^*$ yields that for all $x\in (0,\bar{x})$
\begin{align*}
F^{*k}(x)&\le \frac{x^k}{k!}f(x/k)^k =\frac{x^k}{k!} \Big[\frac{\rho \alpha k}{x} (\log \tfrac{k}{x})^{\alpha-1}\exp(-\rho (\log \tfrac{k}{x})^{\alpha})\Big]^k\\
&= \frac{(\rho \alpha k)^k}{k!} (\log \tfrac{k}{x})^{k(\alpha-1)} \exp(-k\rho (\log \tfrac{k}{x})^{\alpha}).
\end{align*}
By Stirling's estimate $k!\ge e^{-k}k^k$. Now the claim is established for $\hat{x}=\bar{x}$.
\end{proof}

For compact notation and easier comparison, we denote
\[
s_n:= \frac{(\log n)^{-1+1/\alpha}}{\alpha \rho^{1/\alpha}}, \quad k_n^-:=\lfloor s_n \log n\rfloor, \quad  k_n^+:=\lceil s_n \log n \rceil.
\]
In particular, we have
\begin{equation}\label{conversion_eq}
u_n=e^{-\alpha s_n \log n}, \quad \rho (\alpha s_n \log n)^\alpha =\log n \quad \text{and} \quad k_n =O((\log n)^{1/\alpha}).
\end{equation}
Statement (\ref{weight_alpha_1}) of \refthm{weight_a_thm} will follow from the upper and lower bounds established in Theorems \ref{upper_a_thm} and \ref{lower_a_thm}.

\begin{theorem}\label{upper_a_thm}{\bf{(Upper bound for $W_n$ in the $e^{-(E/\rho)^{1/\alpha}}$ case)}}
Let $X_e\overset{d}{=}e^{-(E/\rho)^{1/\alpha}}$, $\rho>0$ and $\alpha>2$. Then
\[
W_n \le (1+\epsilon_n) ek_n^+u_n  \qquad \text{with high probability}
\]
for every sequence $(\eps_n)_{n \in \N}\in (0,\infty)^{\N}$ which satisfies the conditions $(\log n)^{-1+2/\alpha} \log \log n=o(\eps_n)$ and $(\log n)^{-1/\alpha} (\log \log n)^3=o(\eps_n)$.
\end{theorem}

Note that the first condition for $\eps_n$ implies the second if $\alpha <3$ and the second implies the first when $\alpha \ge 3$. Both conditions hold for a constant $\eps_n =\eps$ if $\alpha>2$. For the proof we write $k_n :=k_n^+$ and $b_n := (1+\eps_n)e k_n u_n$.

\begin{proof}[Proof of \refthm{upper_a_thm}] Without loss of generality, assume that $(\eps_n)$ is a null sequence. To apply \refprop{gen_upper_prop}, we first show the divergence of $\E[N_{k_n}(b_n)]$ to infinity. By (\ref{mean_estimate}) and \reflemma{Fk_estimate_lem}, for all sufficiently large $n$
\begin{align*}
\E[N_{k_n}(b_n)]&\ge (1+o(1)) n^{k_n-1} \exp\big(- k_n \rho(\log \tfrac{k_n}{b_n})^{\alpha}\big)\\
&=(1+o(1))\exp\big((k_n-1)\log n-k_n \rho\big(\alpha s_n \log n-\log((1+\eps_n)e)\big)^{\alpha}\big).
\end{align*}
Using a Taylor expansion and (\ref{conversion_eq}), the exponent equals
\begin{align*}
&(k_n-1) \log n - k_n\rho (\alpha s_n \log n)^{\alpha} \Big[1-\frac{\log((1+\eps_n)e)}{s_n \log n}+O((\log n)^{-2/\alpha}) \Big]\\
&\;=\log n \big[-1+ \frac{k_n}{s_n \log n} \log(e(1+\eps_n)) +O((\log n)^{-1/\alpha})\big]\ge \eps_n \log n \big[1+o(1)+O(\eps_n^{-1} (\log n)^{-1/\alpha})\big].
\end{align*}
For the inequality we used that $k_n=k_n^+$. The second condition on $(\eps_n)$ implies that $\eps_n^{-1} (\log n)^{-1/\alpha}=o(1)$ and $\eps_n \log n \to \infty$ for $n \to \infty$. Hence, $\E[N_{k_n}(b_n)]\to \infty$. For the variance bound in \refprop{gen_upper_prop}, we use \reflemma{Fk_estimate_lem} twice to estimate
\[
\frac{F^{*(k_n-l)}(b_n)}{F^{*k_n}(b_n)}\le \frac{d_l}{l+1}\exp\left(-(k_n-l) \rho(\log \tfrac{k_n-l}{b_n})^{\alpha}+k_n \rho(\log \tfrac{k_n}{b_n})^{\alpha}\right),
\]
where $d_l:=d(\alpha,b_n,k_n,l):= (l+1)(\rho \alpha e)^{k_n-l}(\log \frac{k_n-l}{b_n})^{(k_n-l)(\alpha-1)}$. Hence, it is sufficient to show that the following sequence tends to zero:
\begin{equation}\label{sum_to_zero}
\sum_{l=1}^{k_n-2} d_l\exp\left((k_n-l)\rho [(\log \tfrac{k_n}{b_n})^{\alpha}-(\log \tfrac{k_n-l}{b_n})^{\alpha}]+l\rho[(\log \tfrac{k_n}{b_n})^{\alpha}-\rho^{-1}\log n]\right).
\end{equation}
We first derive asymptotics for the expression in the exponent. A Taylor expansion and (\ref{conversion_eq}) yield
\begin{align*}
&(\log \tfrac{k_n-l}{b_n})^{\alpha}=(\log \tfrac{k_n}{b_n})^{\alpha}\Big(1-\frac{\log \frac{k_n}{k_n-l}}{\log \frac{k_n}{b_n}}\Big)^{\alpha}\\
&\phantom{(\log \tfrac{k_n-l}{b_n})^{\alpha}}= (\log \tfrac{k_n}{b_n})^{\alpha}\Big[1-\alpha \frac{\log \frac{k_n}{k_n-l}}{\log \frac{k_n}{b_n} }+\tfrac{1}{2}\alpha(\alpha-1) \left(\frac{\log \frac{k_n}{k_n-l}}{\log \frac{k_n}{b_n}}\right)^2+O\left(\left(\frac{\log \frac{k_n}{k_n-l}}{\log \frac{k_n}{b_n}}\right)^3\right)\Big]\\
&\rho^{-1}\log n=(\log(\tfrac{k_n}{b_n}e(1+\eps_n)))^{\alpha}\\
&\phantom{\rho^{-1} \log n}= (\log\tfrac{k_n}{b_n})^{\alpha} \Big[1+\alpha \frac{\log(e(1+\eps_n))}{\log \tfrac{k_n}{b_n}}+\tfrac{1}{2}\alpha(\alpha-1)\left(\frac{\log(e(1+\eps_n))}{\log \tfrac{k_n}{b_n}}\right)^2+O((\log \tfrac{k_n}{b_n})^{-3})\Big].
\end{align*}
Notice that in particular, $\log \frac{k_n}{b_n}=O((\log n)^{1/\alpha})$. Since the second order terms have a positive coefficient in both cases, we can bound the exponent in (\ref{sum_to_zero}) by
\begin{align*}
&\rho \alpha (\log \tfrac{k_n}{b_n})^{\alpha-1} \Big[ (k_n-l) \Big(\log \tfrac{k_n}{k_n-l}+O((\log k_n)^3 (\log \tfrac{k_n}{b_n})^{-2})\Big)- l \Big(\log(e (1+\eps_n)) +O((\log \tfrac{k_n}{b_n})^{-2})\Big)\Big]\\
&\quad=\rho \alpha (\log \tfrac{k_n}{b_n})^{\alpha-1} l \Big[ (\tfrac{k_n}{l}-1) \log \tfrac{k_n}{k_n-l} - \log(e (1+\eps_n)) + O((\tfrac{k_n}{l}-1)(\log k_n)^3 (\log \tfrac{k_n}{b_n})^{-2}) \Big].
\end{align*}
By (\ref{logpn_control}), (\ref{conversion_eq}) and the second assumption on $(\eps_n)$,
\[
O\big((\tfrac{k_n}{l}-1)(\log k_n)^3 (\log \tfrac{k_n}{b_n})^{-2}\big)= O\big(k_n(\log \log n)^3 (\log n)^{-2/\alpha}\big)=o(\eps_n).
\]
Hence, the second half of the bracket equals $-\eps_n (1+o(1))$; the first part equals $\psi(l/k_n)-1$, where $\psi(q)= (\frac{1}{q} -1)\log (1-q)^{-1}$. According to \reflemma{psi_lemma}, $\psi$ is bounded by one. Using the definition of $d_l$, we can estimate
\begin{align*}
\sum_{l=1}^{k_n-2} d_l &\exp\Big(-\rho \alpha (\log \tfrac{k_n}{b_n})^{\alpha -1} l \eps_n (1+o(1))\Big)\\
&\le \sum_{l=1}^{k_n-2} (l+1)\exp\Big(-l \Big[ \rho \alpha (\log \tfrac{k_n}{b_n})^{\alpha-1}\eps_n(1+o(1)) + k_n \log(\rho \alpha e)+k_n(\alpha-1) \log \log \tfrac{k_n}{b_n}\Big]\Big).
\end{align*}
The two last terms in the exponent are of order $O((\log n)^{1/\alpha}\log \log n)=o(\eps_n (\log \frac{k_n}{b_n})^{\alpha-1})$ by the first assumption on $(\eps_n)$. Since $\eps_n (\log \tfrac{k_n}{b_n})^{\alpha-1}\to \infty$, the sum converges to zero. Thus, both conditions of \refprop{gen_upper_prop} are satisfied and the claim follows.
\end{proof}

Recall that $k_n^- = \lfloor s_n \log n\rfloor$. For the lower bound it is sufficient to have $\alpha>1$:

\begin{theorem}\label{lower_a_thm}{\bf{(Lower bound for $W_n$ in the $e^{-(E/\rho)^{1/\alpha}}$ case)}}
Let $X_e\overset{d}{=}e^{-(E/\rho)^{1/\alpha}}$, $\rho>0$ and $\alpha>1$. Then
\[
 W_n \ge (1-\eps_n)ek_n^-u_n  \qquad \text{with high probability,}
\]
for every sequence $(\eps_n)_{n \in \N} \in (0,\infty)^{\N}$ which satisfies the conditions $(\log n)^{-1+1/\alpha} \log \log n =o(\eps_n)$ and $(\log n)^{-1/\alpha} (\log \log n)^2=o(\eps_n)$.
\end{theorem}

Note that in the case $\alpha \in (1,2)$ the first assumption is stronger and in the case $\alpha \ge 2$ the second. The first condition agrees with the condition found for edge weights $X_e\overset{d}{=}Z^{s_n}$, the second one is needed to estimate the distribution function. In the proof we write $d_n=(1-\eps_n)ek_n^-u_n$.

\begin{proof}[Proof of \refthm{lower_a_thm}]
Without loss of generality, assume that $(\eps_n)$ is a null sequence. We check the criterion of \refprop{lower_gen_prop}. By \reflemma{Fk_estimate_lem},
\[
\sum_{k=1}^{n-1} n^{k-1} F^{*k}(d_n)\le \sum_{k=1}^{n-1} (\rho \alpha e)^k (\log \tfrac{k}{d_n})^{k(\alpha-1)}\exp(-k\rho (\log \tfrac{k}{d_n})^{\alpha}+(k-1)\log n).
\]
For the factor $(\log \frac{k}{d_n})^{k(\alpha-1)}$ we use the rough bound $(2 \log n)^{k \alpha}$. The idea for showing that the sum vanishes is to consider small and large $k$ separately. Let $\kappa_n:=\lceil 2e k_n^-\rceil$. We start with $k\in [\kappa_n]$. This allows us to use similar asymptotics as in the proof of the upper bound. A Taylor expansion and the second assumption on $(\eps_n)$ yield
\begin{align*}
(\log \tfrac{k}{d_n})^{\alpha}&=(\alpha s_n \log n)^{\alpha} \Big[1- \frac{\log(\frac{k_n^-}{k} (1-\eps_n)e)}{s_n \log n} + o((\log n)^{-1/\alpha}\eps_n)\Big]\\
&= -\rho^{-1}\log n \frac{1}{s_n \log n} \Big[\log\big(\tfrac{k_n^-}{k}(1-\eps_n)e n^{-s_n}\big)+o(\eps_n)\Big].
\end{align*}
We can now take advantage of our computations for the $Z^{s_n}$ case. By (\ref{conversion_eq}) and (\ref{exponent_lower_bound}) without the error term,
\begin{equation}\label{exp_exp_eq}
-k \rho (\log \tfrac{k}{d_n})^{\alpha}+(k-1)\log n =kp_n \big[\varphi(\tfrac{s_n \log n}{k})+1+ \log(1-\eps_n) + \log \tfrac{k_n^-}{s_n \log n} +o(\eps_n)\big].
\end{equation}
Since $\varphi(x)= -x+\log x \le-1$, we obtain
\[
\sum_{k=1}^{\kappa_n} n^{k-1} F^{*k}(d_n) \le \sum_{k=1}^{\infty} \exp\left(-k \Big[ \eps_n p_n (1+o(1)) - \log(\rho \alpha e)- \alpha \log(2 \log n)\Big]\right).
\]
By the first assumption on $(\eps_n)$ the exponent equals $-k \eps_n p_n (1+o(1))$ and $\eps_n p_n \to \infty$ for $n \to \infty$. This shows that the sum over $k\leq \kappa_n$ vanishes. In the next step we consider $k>\kappa_n$:
\[
(\log \tfrac{k}{d_n})^{\alpha}\ge \log(2e^{\alpha s_n \log n})^{\alpha} = (\alpha s_n \log n)^{\alpha} \Big[1+ \frac{\log 2}{s_n \log n} + O((\log n)^{-2/\alpha})\Big].
\]
Thus, the term in the exponent turns into
\[
-k\rho (\log \tfrac{k}{d_n})^{\alpha}+(k-1)\log n\le -k p_n [\log 2 +o(1)].
\]
Since $\log(\rho \alpha e)+\alpha \log(2 \log n) =o(p_n)$, we derive
\[
\sum_{k=\kappa_n+1}^{n-1} n^{k-1} F^{*k}(d_n)\le \sum_{k=1}^{\infty} \exp\big(-kp_n [\log 2+o(1)]\big)=o(1).
\]
The claim now follows from \refprop{lower_gen_prop}.
\end{proof}

We turn to the hopcount. Statement (\ref{hop_alpha_thm}) in \refthm{weight_a_thm} is an immediate consequence of the following result:

\begin{theorem}{\bf{(Hopcount in the $e^{-(E/\rho)^{1/\alpha}}$ case)}}
Let $X_e \overset{d}{=} e^{-(E/\rho)^{1/\alpha}}$ for $\rho >0$, $\alpha >2$. Then
\[
\frac{1}{1+\beta_n}s_n \log n< H_n < \frac{1}{1-\beta_n}s_n \log n \qquad \text{with high probability}
\]
for all $(\beta_n)_{n \in \N}\in(0,\infty)^{\N}$ with $(\log n)^{-1/\alpha} (\log \log n)^4=o(\beta_n^2)$ and $(\log n)^{-1+2/\alpha}(\log \log n)^2 =o(\beta_n^2)$ for $n \to \infty$.
\end{theorem}

\begin{proof}
Let $\eps_n=\max\{(\log n)^{-1/\alpha}(\log \log n)^4, (\log n)^{-1+2/\alpha}(\log \log n)^2\}$ and $b_n=(1+\eps_n)ek_n^+ u_n$. Notice that $\eps_n$ satisfies the requirements of Theorems \ref{upper_a_thm} and \ref{lower_a_thm} and $\eps_n=o(\beta_n^2)$. Without loss of generality, $(\beta_n)$ is a null sequence. To prove the lower bound for the hopcount, denote $h_n=\lfloor \frac{1}{1+\beta_n}s_n\log n\rfloor$. By \refthm{upper_a_thm} and (\ref{hopcount_lower_gen}), it is sufficient to show that $\sum_{k=1}^{h_n} \E[N_k(b_n)]\to 0$ for $n \to \infty$. This sum was already considered in the proof of \refthm{lower_a_thm} with $d_n$ in place of $b_n$ and $h_n$ replaced by $\kappa_n=\lceil 2ek_n^- \rceil$. Since only the order of $\kappa_n$ mattered for the summands with small $k$, the estimates for the summands are still valid when $k_n^-$ is replaced by $k_n^+$ and $1-\eps_n$ by $1+\eps_n$. By (\ref{exp_exp_eq}) we can consider
\[
\sum_{k=1}^{h_n}(\rho \alpha e)^k (2\log n)^{k\alpha}\exp(kp_n\big[\varphi(\tfrac{s_n \log n}{k})+1+\log(1+\eps_n)+\log \tfrac{k_n^+}{s_n\log n}+o(\eps_n)\big]),
\]
where $\varphi(x)=-x+\log x$. As in the proof of \refthm{hop_sn_thm}, $\varphi(\frac{s_n\log n}{k})\le\varphi(1+\beta_n)=-1-\frac{1}{2}\beta_n^2(1+o(1))$. Since by assumption $\eps_n=o(\beta_n^2)$ and $\beta_n^2 s_n \log n \to \infty$, we can bound the sum by
\begin{equation}\label{control_sum}
\begin{split}
\sum_{k=1}^{\infty} \exp\Big(k\Big[-\tfrac{1}{2}p_n\beta_n^2 (1+o(1)) + \log(\alpha \rho e)+& \alpha \log(2 \log n)\Big]\Big) \\
&=\sum_{k=1}^{\infty} \exp\big(-\tfrac{1}{2}p_n\beta_n^2 k(1+o(1))\big)=o(1).
\end{split}
\end{equation}
Here we used the definition of $s_n$ and the second assumption on $\beta_n$. The lower bound is thus proven.

For the upper bound, we set $h_n=\lceil \frac{1}{1-\beta_n}s_n \log n\rceil$. It suffices to show that $\sum_{k=h_n}^{n-1} \E[N_k(b_n)]\to 0$ for $n \to \infty$. Splitting this sum into the summands with $k \le \kappa_n:=\lceil 4ek_n^+ \rceil$ and $k > \kappa_n$, the summands with large $k$ can be handled in the same way as in the proof of \refthm{lower_a_thm}. The $2$ was replaced by $4$ to make up for the $1+\eps_n$ in $b_n$ instead of $1-\eps_n$ in $d_n$. For the summands $k \in \{h_n,\dotsc,\kappa_n\}$ we can use the same estimates as in the first part of the current proof and are left with estimating
\[
\sum_{k=h_n}^{\kappa_n} (\rho \alpha e)^k (2\log n)^{k\alpha} \exp(kp_n\big[\varphi(\tfrac{s_n \log n}{k})+1+\log(1+\eps_n)+\log \tfrac{k_n^+}{s_n\log n}+o(\eps_n)\big]).
\]
Since $\varphi(\frac{s_n \log n}{k})\le \varphi(1-\beta_n)=-1-\frac{1}{2}\beta_n^2 (1+o(1))$ the sum converges to zero as shown in (\ref{control_sum}).
\end{proof}

{\bf{Acknowledgements:}} A substantial part of this work has been done at Eurandom and Eindhoven University of Technology. ME and JG are grateful to both institutions for their hospitality.\\
The work of JG was supported in part by the European Research Council. The work of RvdH was supported in part by the Netherlands Organisation for Scientific Research (NWO).

\end{document}